\documentclass[12pt]{amsart}

\usepackage{amscd}
\usepackage{amsmath}
\usepackage{amsfonts}
\usepackage{amssymb}
\usepackage{graphics}
\usepackage{epsfig}

\font\Bbb=msbm10 scaled \magstep 2

\def\C{\hbox{\Bbb C}}
\def\R{\hbox{\Bbb R}}
\def\Z{\hbox{\Bbb Z}}
\def\N{\hbox{\Bbb N}}

\font\midBbb=msbm8

\def\midR{\hbox{\midBbb R}}
\def\midZ{\hbox{\midBbb Z}}
\def\midN{\hbox{\midBbb N}}

\newtheorem{theorem}{\bf Theorem}
\newtheorem{lemma}[theorem]{\bf Lemma}
\newtheorem{remark}[theorem]{\bf Remark}
\newtheorem{example}[theorem]{\bf Example}

\newtheorem{corollary}[theorem]{\bf Corollary}
\newtheorem{definition}[theorem]{\bf Definition}

\numberwithin{equation}{section} \numberwithin{theorem}{section}

\title{Hypergeometric polynomials are optimal}

\author{D.V.\,Bogdanov and T.M.\,Sadykov}
\address{Department of Mathematics
\newline \indent and Computer Science,
\newline \indent Plekhanov Russian University
\newline \indent 125993, Moscow, Russia.}
\email{Sadykov.TM@rea.ru}
\thanks{The second author was supported by
the grant of the President of the Russian Federation for state
support of leading scientific schools NSh-9149.2016.1 and the
grant of the Government of the Russian Federation for
investigations under the guidance of the leading scientists of the
Siberian Federal University (contract No.~14.Y26.31.0006)}

\begin{document}

\begin{abstract}
With any integer convex polytope $P\subset\midR^n$ we associate a
multivariate hypergeometric polynomial whose set of exponents is
$\midZ^{n}\cap P.$ This polynomial is defined uniquely up to a
constant multiple and satisfies a holonomic system of partial
differential equations of Horn's type. We prove that under certain
nondegeneracy conditions the zero locus of any such polynomial is
optimal in the sense of~\cite{FPT}, i.e., that the topology of its
amoeba~\cite{Mikhalkin} is as complex as it could possibly be.
Using this, we derive optimal properties of several classical
families of multivariate hypergeometric polynomials.
\end{abstract}

\maketitle

\section{Introduction
\label{sec:introduction}}

Zeros of hypergeometric functions are known to exhibit highly
complicated behavior. The univariate case has been extensively
studied both classically (see e.g. \cite{Klein,Norlund}) and
recently (see~\cite{Dominici,DriverJohnston,ZhouSrivastavaWang}
and the references therein). Already the distribution of zeros of
polynomial instances of the simplest non-elementary hypergeometric
function $_{2}F_{1}\left(a,b;c;x\right)$ is far from being clear.
By letting the parameters $a,b,c$ assume values in various ranges,
one can obtain a wide variety of shapes. Some of them are highly
regular (see e.g. Fig.~\ref{fig:HGaster}) while other are nearly
chaotic.

\begin{figure}[ht!]
\centering{}\includegraphics[width=5cm]{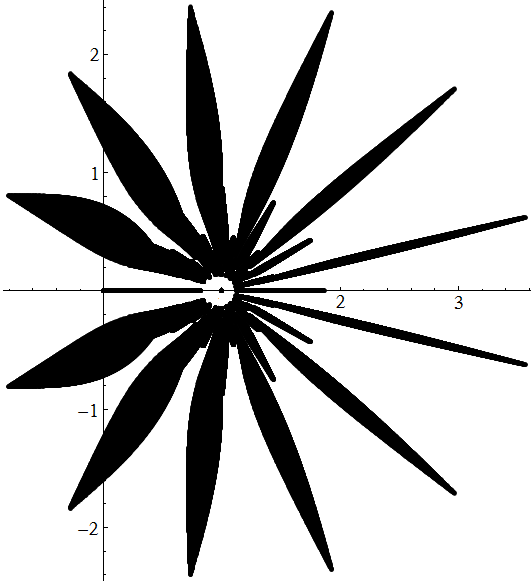} \caption{The
hypergeometric aster: zeros of the polynomials
$_{2}F_{1}\left(-12,b;c;x\right)$ with $b,c\in \left\{
\frac{k}{1000} : k=100,\ldots, 4000 \right\}$} \label{fig:HGaster}
\end{figure}

Polynomial instances of hypergeometric functions in one and
several variables are very diverse. They comprise the classical
Chebyshev polynomials of the first and the second kind, the
Gegenbauer, Hermite, Jacobi, Laguerre and Legendre polynomials as
well as their numerous multivariate analogues~\cite{DunklXu}.

Despite the diversity of families of hypergeometric polynomials,
most of them share the following key properties that justify the
usage of the term "hypergeometric":

1. The polynomials are dense (possibly after a suitable monomial
change of variables).

2. The coefficients of a hypergeometric polynomial are related
through some recursion with polynomial coefficients.

3. For univariate polynomials, there is typically a single
representative (up to a suitable normalization) of a given degree
within a family of hypergeometric polynomials.

4. All polynomials in the family satisfy a differential equation
of a fixed order with polynomial coefficients (or a system of such
equations) whose parameters encode the degree of a polynomial.

5. In the case of one dimension, the absolute values of the roots
of a classical hypergeometric polynomial are all different
(possibly after a suitable monomial change of variables).

6. Many of hypergeometric polynomials enjoy various extremal
properties.

In the present paper, we introduce a definition of a multivariate
hypergeometric polynomial in $n\geq 2$ complex variables that is
coherent with the properties 1-6 listed above. Namely, with any
integer convex polytope $P\subset\R^n$ we associate a multivariate
hypergeometric polynomial whose set of exponents is $\Z^{n}\cap
P.$ For this polynomial to be truly hypergeometric in the sense
made precise below, we need to assume that any pair of points in
$\Z^{n}\cap P$ can be connected by a polygonal line with unit
sides and integer vertices. This assumption does not affect the
generality of the results since any polytope that does not satisfy
this condition gives rise to a finite number of hypergeometric
polynomials that can be considered independently. The assumption
of convexity of the polytope~$P$ is then automatically satisfied
as clarified in Lemma~\ref{lem:HGpolyHasConvSupp}.

The hypergeometric polynomial associated with the polytope~$P$ is
defined uniquely up to a constant multiple and satisfies a
holonomic system of partial differential equations of Horn's
type~\cite{Sadykov-SMZh,SadykovTanabe}. We prove that under
certain nondegeneracy conditions (see
Theorem~\ref{thm:HGpolyHasOptimalAmoeba}) the zero locus of any
such polynomial is optimal in the sense of~\cite{FPT}. Generally
speaking, this means that the topology of the
amoeba~\cite{FPT,Mikhalkin} of such a polynomial is as complicated
as it could possibly be (see Definition~\ref{def:optimalAmoeba}).
This property is the multivariate counterpart of the property of
having different absolute values of the roots for a polynomial in
a single variable. We show various families of classically known
multivariate polynomials to be optimal: a biorthogonal basis in
the unit ball, certain polynomial instances of the Appel~$F_1$
function, bivariate Chebyshev polynomials of the second kind etc.
We also introduce a simplicial complex that encodes intrinsic
combinatorial properties of an algebraic variety while possessing
key properties of its compactified amoeba.

Pictures of amoebas in the paper have been created with
Matlab~7.9. The authors are thankful to L.\,Lang for valuable
comments on the paper.


\section{Hypergeometric systems and amoebas
\label{sec:notation}}

Throughout the paper, we denote by $n$ the number of $x\in\C^n$
variables. For $\alpha=(\alpha_1,\ldots,\alpha_n),$ we use the
notation $|\alpha|=\sum_{i=1}^n \alpha_i$ and $\alpha!=\alpha_{1}!
\ldots \alpha_{n}!.$ For $x=(x_1,\ldots,x_n)$ and
$\alpha=(\alpha_1,\ldots,\alpha_n),$ we denote by $x^\alpha$ the
monomial $x_{1}^{\alpha_1}\ldots x_{n}^{\alpha_n}.$

\begin{definition}\label{def:HGfunction}
\rm A formal Laurent series
\begin{equation}
\sum_{s\in\midZ^n} \varphi(s) \, x^{s} \label{series}
\end{equation}
is called {\it hypergeometric} if for any $j=1,\ldots,n$ the
quotient $\varphi(s+e_{j})/\varphi(s)$ is a rational function
in~$s = (s_1,\ldots,s_n).$ Throughout the paper we denote this
rational function by $P_{j}(s)/Q_{j}(s+e_{j}).$
Here~${\{e_j\}}_{j=1}^{n}$ is the standard basis of the
lattice~$\Z^n.$ By the {\it support} of this series we mean the
subset of~$\Z^n$ on which $\varphi(s)\neq 0.$
\end{definition}

By a {\it hypergeometric function} we will mean a (typically
multi-valued) analytic function obtained by means of analytic
continuation of a hypergeometric series with a nonempty domain of
convergence along all possible paths in~$\C^n$.

\begin{theorem} {\rm (Ore, Sato, see \cite{GGR}.)} The
coefficients of a hypergeometric series are given by the formula
\begin{equation}
\varphi(s) = t^{s} \, U(s) \, \prod_{i=1}^{m} \Gamma(\langle {\bf
A}_{i}, s \rangle + c_{i}), \label{oresatocoeff}
\end{equation}
where $t^s = t_{1}^{s_1}\ldots t_{n}^{s_n},$ $t_i, c_i\in\C,$
${\bf A}_i=(A_{i,1}, \ldots, A_{i,n})$ $\in\Z^n,$ $i =1, \ldots,
m,$ and $U(s)$~is the product of a certain rational function and a
periodic function $\phi(s)$ such that $\phi(s+e_j)\equiv\phi(s)$
for every $j = 1,\ldots,n$. \label{thm:Ore-Sato}
\end{theorem}

Given the above data ($t_i, c_i, {\bf A}_i, U(s)$) that determines
the coefficient of a hypergeometric series, it is straightforward
to compute the rational functions $P_{i}(s)/Q_{i}(s+e_{i})$ using
the $\Gamma$-function identity. The converse requires solving a
system of difference equations which is only solvable under some
compatibility conditions on $P_i,Q_i.$ A careful analysis of this
system of difference equations has been performed
in~\cite{Sadykov-SMZh}.

We will call any function of the form~(\ref{oresatocoeff}) {\it
the Ore-Sato coefficient of a hypergeometric series.} In this
paper the Ore-Sato coefficient~(\ref{oresatocoeff}) plays the role
of a primary object which generates everything else: the series,
the hypergeometric system of differential equations, its
polynomial solution (if any) and its amoeba. We will also assume
that $m \geq n$ since otherwise the corresponding hypergeometric
series~(\ref{series}) is just a linear combination of
hypergeometric series in fewer variables (times an arbitrary
function in remaining variables that makes the system
non-holonomic) and~$n$ can be reduced to meet the inequality.

\begin{definition}\label{def:Horn system}\rm
{\it The Horn system of an Ore-Sato coefficient.} A (formal)
Laurent series $\sum_{s\in\midZ^n}\varphi(s)x^s$ whose coefficient
satisfies the relations $\varphi(s+e_j)/\varphi(s) =
P_j(s)/Q_j(s+e_j)$ is a (formal) solution to the following system
of partial differential equations of hypergeometric type
\begin{equation}
x_j P_j(\theta)f(x) = Q_j(\theta)f(x), \,\,\, j=1,\ldots,n.
\label{horn}
\end{equation}
Here $\theta=(\theta_1,\ldots,\theta_n),$ $\theta_j =
x_j\frac{\partial}{\partial x_j}.$ The system~(\ref{horn}) will be
referred to as {\it the Horn hypergeometric system defined by the
Ore-Sato coefficient~$\varphi(s)$} (see~\cite{GGR}) and denoted by
${\rm Horn}(\varphi)$. In this paper we only treat holonomic Horn
hypergeometric systems, i.e. $ {\rm rank} ({\rm Horn}(\varphi))$
is always assumed to be finite.
\end{definition}
We will often be dealing with the important special case of an
Ore-Sato coefficient~(\ref{oresatocoeff}) where $t_i=1$ for any
$i=1,\ldots,n$ and $U(s)\equiv 1.$ The Horn system associated with
such an Ore-Sato coefficient will be denoted by ${\rm Horn}(A,c),$
where~$A$ is the matrix with the rows ${\bf A}_1,\ldots, {\bf A}_m
\in \Z^n$ and $c=(c_1, \ldots, c_m)\in\C^m.$ In this case the
following operators $P_j(\theta)$ and $Q_j(\theta)$ explicitly
determine the system~(\ref{horn}):
$$ P_j(s) = \prod_{i : A_{i,j}>0} \prod_{\ell_j^{(i)}=0}^{A_{i,j}-1} \left( \langle {\bf A}_{i}, s
\rangle + c_{i} + \ell_j^{(i)}\right), \quad Q_j(s) = \prod_{i :
A_{i,j}<0} \prod_{\ell_j^{(i)}=0}^{|A_{i,j}| -1} \left( \langle
{\bf A}_{i}, s \rangle + c_{i} + \ell_j^{(i)}\right). $$

\begin{definition}
\rm The support of a series solution to~(\ref{horn}) is called
{\it irreducible} if there exists no series solution
to~(\ref{horn}) supported in its proper nonempty subset.
\label{def:irreducibleSupport}
\end{definition}

\begin{definition}
\label{def:amoeba} \rm The {\it amoeba}~$\mathcal{A}_f$ of a
Laurent polynomial~$f(x)$ (or of the algebraic hypersurface $\{
f(x)=0 \}$) is defined to be the image of the
hypersurface~$f^{-1}(0)$ under the map ${\rm Log } :
(x_1,\ldots,x_n)\mapsto (\log |x_1|,\ldots,\log |x_n|).$
\end{definition}

Despite losing~$n$ real dimensions, the amoeba of an algebraic
hypersurface encodes its several intrinsic
properties~\cite{FPT,Purbhoo,Theobald}. The main results of the
paper describe topological properties of the amoebas of
hypergeometric polynomials. The next lemma shows that certain
transformations of a polynomial do not affect the topology of its
amoeba.

\begin{lemma}\label{lem:transfThatKeepAmoebaProp}
The number of connected components of the amoeba complement of a
(Laurent) polynomial $p(x_1,\ldots,x_n)$ is the same as that of
the polynomial $x^{a} p(t_1 x^{v_1},\ldots, t_n x^{v_n})^{\ell}$
for any $\ell\in\N,$ $a=(a_1,\ldots,a_n)\in\Z^n,$ $t=(t_1,\ldots,
t_n)\in(\C^{*})^{n}$ and any nondegenerate integer matrix $v$ with
the rows $v_1,\ldots,v_n.$ That is, there is a bijection between
the connected components of the complements of the two amoebas;
moreover, the orders~\cite{FPT} and the recession cones~\cite{PST}
of the corresponding components are transformed into each other by
the linear map with the matrix~$v$.
\end{lemma}
\begin{proof}
A monomial factor can only vanish in the union of the coordinate
hyperplanes that is mapped off to infinity by the logarithmic map.
The amoeba does not reflect the multiplicities of the zeros of a
polynomial and it can therefore be raised to any positive power.
The map $(x_1,\ldots,x_n)\mapsto (t_1 x_1,\ldots, t_n x_n)$
corresponds to the shift of the amoeba space with respect to the
vector $\log|t_1|,\ldots,\log|t_n|.$ Finally, a nondegenerate
monomial change of variables in the complex torus $(\C^{*})^n$
corresponds to the linear transformation of the amoeba space
defined by the matrix of exponents of the monomials. Clearly a
nondegenerate linear map preserves the topology of amoebas and
provides a bijection for the recession cones of the complement
components. The last statement of the lemma is an immediate
consequence of the definition of the order of a component in the
amoeba complement~\cite{FPT}.
\end{proof}

Recall that the {\it Newton polytope}~$\mathcal{N}_{p(x)}$ of a
Laurent polynomial~$p(x)$ is defined to be the convex hull
in~$\R^n$ of the support of~$p(x).$ The following result shows
that the Newton polytope~$\mathcal{N}_{p(x)}$ reflects the
structure of the amoeba~$\mathcal{A}_{p(x)}$ \cite[Theorem~2.8 and
Proposition~2.6]{FPT}.

\begin{theorem}\label{3thmfptestimate} {\rm (See \cite{FPT}.)}
Let~$p(x)$ be a Laurent polynomial and let~$\{M\}$ denote the
family of connected components of the amoeba
complement~$^c\!\mathcal{A}_{p(x)}.$ There exists an injective
function $\nu : \{M\}\rightarrow \Z^n \cap \mathcal{N}_{p(x)}$
such that the cone which is dual to~$\mathcal{N}_{p(x)}$ at the
point~$\nu(M)$ coincides with the recession cone~\cite{PST}
of~$M.$ In particular, the number of connected components of
$^c\!\mathcal{A}_{p(x)}$ cannot be smaller than the number of
vertices of $\mathcal{N}_{p(x)}$ and cannot exceed the number of
integer points in $\mathcal{N}_{p(x)}.$
\end{theorem}

The two extreme values for the number of connected components of
the complement of an amoeba are of particular interest.

\begin{definition}\label{def:optimalAmoeba}\rm (Cf.~\cite[Definition~2.9]{FPT}.)
An algebraic hypersurface $\mathcal{H}\subset(\C^{*})^n,$ $n\geq
2,$ is called {\it optimal} if the number of connected components
of its amoeba complement $^c\!\mathcal{A}_{\mathcal{H}}$ equals
the number of integer points in the Newton polytope of the
defining polynomial of~$\mathcal{H}.$ We will say that a
polynomial (as well as its amoeba) is optimal if its zero locus is
an optimal algebraic hypersurface.
\end{definition}

Since the amoeba of a polynomial does not carry any information on
the multiplicities of its roots, any one-dimensional amoeba (which
is just a finite set of distinct points in $a_1,\ldots, a_k\in\R$)
can be treated as the amoeba of the polynomial
$\prod\limits_{j=1}^{k}(x-e^{a_k})$ all of whose roots are
positive and distinct. Thus Definition~\ref{def:optimalAmoeba} is
trivial in the univariate case. The correct extension of
Definition~\ref{def:optimalAmoeba} to one dimension is to say that
all the roots of the polynomial in question have different
absolute values.

\begin{example}\label{ex:HGExampleCorrectingMikhalkin}\rm
In accordance with Definition~\ref{def:Horn system} the Ore-Sato
coefficient
$$
\varphi(s,t) = \Gamma(-s - t + 1)\Gamma(2 s - t - 2)\Gamma(-s + 2
t - 2)
$$
yields the polynomials
$$
\begin{array}{l}
P_{1}(s,t)=(2s-t-2)(2s-t-1), \quad\quad Q_{1}(s,t)=(s-2t+2)
(s+t-1), \\
P_{2}(s,t)=(s-2t+1)(s-2t+2), \quad\quad Q_{2}(s,t)=-(2s-t-2)
(s+t-1).
\end{array}
$$
The corresponding Horn hypergeometric system is given by the
linear differential operators
\begin{equation}
\left\{
\begin{array}{l}
x \, (2\theta_x - \theta_y -2)(2\theta_x - \theta_y -1) - (\theta_x - 2\theta_y + 2)(\theta_x + \theta_y - 1) , \\
y \, (\theta_x - 2 \theta_y + 1)(\theta_x - 2 \theta_y + 2) +
(2\theta_x - \theta_y - 2)(\theta_x + \theta_y - 1).
\end{array}
\right. \label{horn(-1,-1)(2,-1)(-1,2)}
\end{equation}
It is straightforward to check
that~(\ref{horn(-1,-1)(2,-1)(-1,2)}) is satisfied by the
hypergeometric polynomial $p(x,y) = x + y + 6xy + x^2y^2.$

The amoeba of $p(x,y)$ together with its compactified
version~\cite{Mikhalkin} are depicted in
Fig.~\ref{fig:MikhalkinAmoeba}. We remark that the actual shapes
of the amoeba and the compactified amoeba of $p(x,y)$ (with
respect to both directions of the amoeba tentacles and the
complement of its compact version) are rather different from those
given in~\cite{Mikhalkin}.
\end{example}

\begin{figure}[ht!]
\includegraphics[width=5cm]{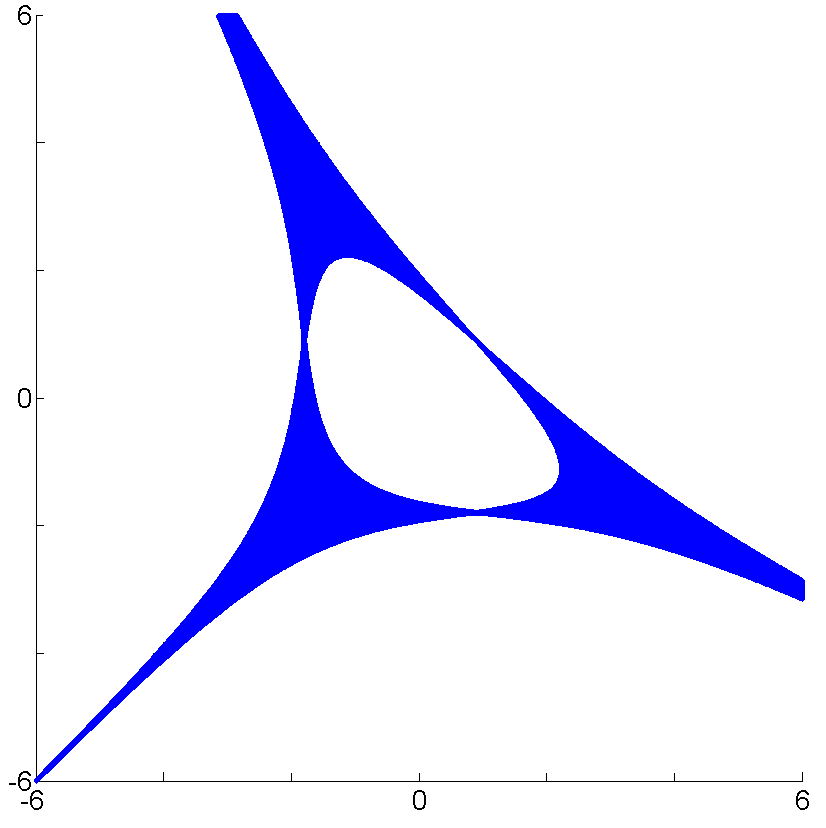}\quad\quad\quad\quad
\includegraphics[width=5cm]{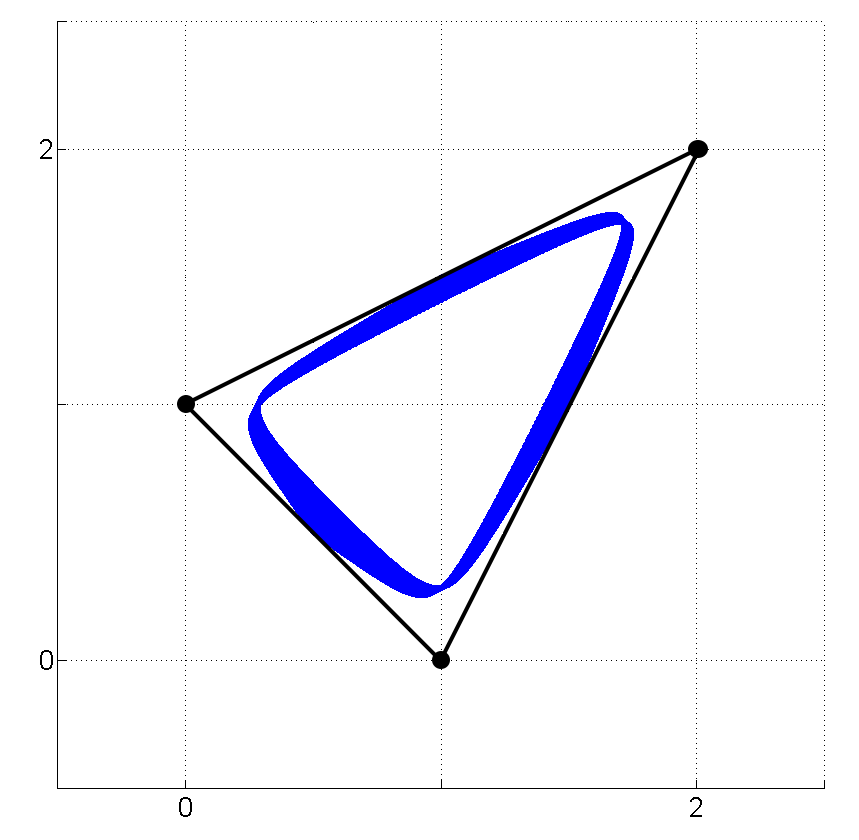}
\caption{The amoeba and the compactified amoeba~\cite{Mikhalkin}
of $p(x,y)$} \label{fig:MikhalkinAmoeba}
\end{figure}

An algebraic hypersurface is optimal if the topology of its amoeba
is as complicated as it could possibly be in the view of
Theorem~\ref{3thmfptestimate} (that is, the number of connected
components in the amoeba complement is maximal). A two-dimensional
optimal amoeba has the maximal possible number of bounded
connected components in its complement and the maximal number of
parallel tentacles.

In the view of Lemma~\ref{lem:transfThatKeepAmoebaProp} we will
not distinguish polynomials whose zero loci in $(\C^{*})^n$ can be
transformed into each other by a nondegenerate monomial change of
variables. The reason for this is illustrated by the following
example.

\begin{example}\label{ex:biorthogonalPolysAreOptimal}\rm
A typical example of a family of optimal polynomials arising in a
different theory is given by the biorthogonal family in the unit
ball $\left\{V_{\alpha}(x)\right\}_{\alpha\in\midN_{0}^{n}},$
$x\in\C^n$ defined through their generating function
(see~\cite[Section~2.3]{DunklXu}):
$$
(1-2\langle a,x \rangle + \|a\|^2)^{\frac{1-n}{2}} =
\sum\limits_{\alpha\in\midN_{0}^{n}} a^{\alpha} V_{\alpha}(x).
$$
Neglecting an inessential monomial factor of $V_{\alpha}(x)$
(whose zero locus is contained in the union of the coordinate
hyperplanes and therefore does not affect the amoeba, see
Lemma~\ref{lem:transfThatKeepAmoebaProp}) one can represent
$V_{\alpha}(x) = \tilde{V}_{\alpha}(\xi)$ with
$\tilde{V}_{\alpha}$ being a polynomial in $\xi_j=x_{j}^{2}.$ One
can check that the zero locus of $\tilde{V}_{\alpha}(\xi)$ is an
optimal algebraic hypersurface in~$\C^{n}.$ The Newton polygon and
the amoeba of the bivariate polynomial~$\tilde{V}_{(6,10)}$ are
depicted in Fig.~\ref{fig:V(6,10)inDunklXu}.

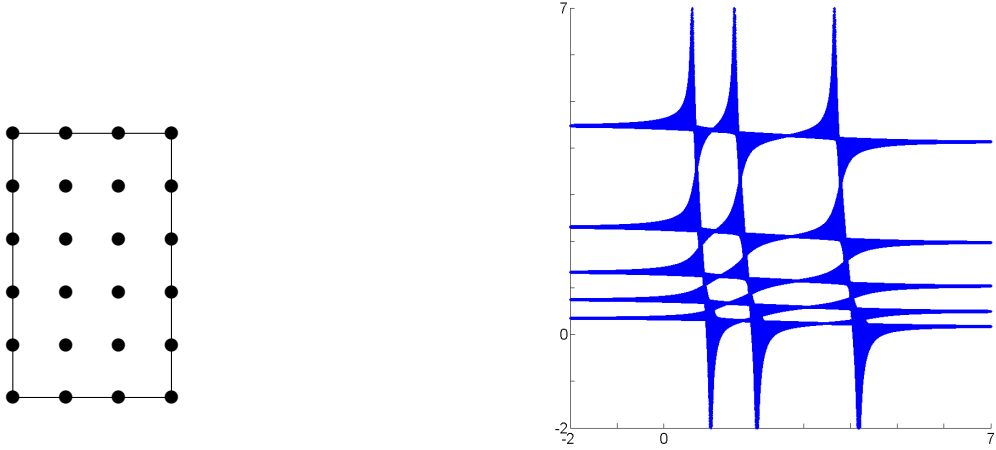
\begin{figure}[ht!]\hskip1cm
\begin{picture}(200,160)

\put(0,20){\circle*{5}} \put(20,20){\circle*{5}}
\put(40,20){\circle*{5}} \put(60,20){\circle*{5}}

\put(0,40){\circle*{5}} \put(20,40){\circle*{5}}
\put(40,40){\circle*{5}} \put(60,40){\circle*{5}}

\put(0,60){\circle*{5}} \put(20,60){\circle*{5}}
\put(40,60){\circle*{5}} \put(60,60){\circle*{5}}

\put(0,80){\circle*{5}} \put(20,80){\circle*{5}}
\put(40,80){\circle*{5}} \put(60,80){\circle*{5}}

\put(0,100){\circle*{5}} \put(20,100){\circle*{5}}
\put(40,100){\circle*{5}} \put(60,100){\circle*{5}}

\put(0,120){\circle*{5}} \put(20,120){\circle*{5}}
\put(40,120){\circle*{5}} \put(60,120){\circle*{5}}

\put(0,20){\line(1,0){60}}\put(0,20){\line(0,1){100}}
\put(0,120){\line(1,0){60}}\put(60,120){\line(0,-1){100}}

\end{picture} \includegraphics[width=6cm]{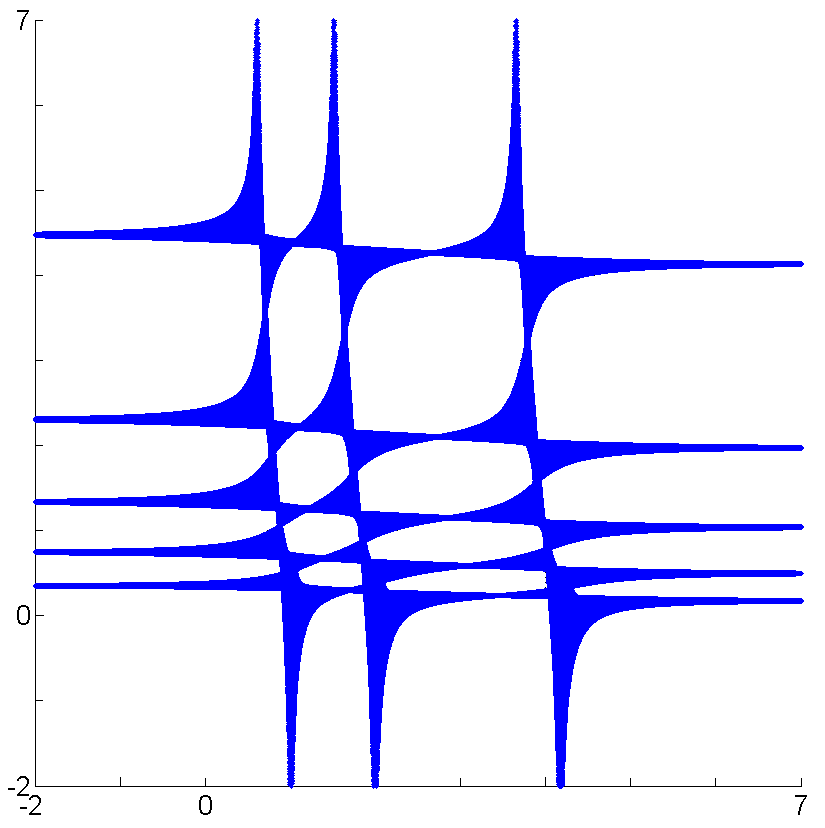}
\caption{The Newton polygon of $\tilde{V}_{(6,10)}(x,y)$
in~\cite[Section~2.3]{DunklXu} and its amoeba}
\label{fig:V(6,10)inDunklXu}
\end{figure}
\noindent Numerous other families of optimal multivariate
polynomials can be found in~\cite[Chapter~2]{DunklXu}.
\end{example}


\section{Hypergeometric polynomials in several variables
\label{sec:HGPolynomials}}

There exist several well-known and important families of
hypergeometric polynomials~\cite{DunklXu,ZhouSrivastavaWang}. What
they all have in common is the fact that they satisfy certain
linear differential equations with polynomial coefficients that
are special instances of~(\ref{horn}). However, the family of all
holonomic systems of partial differential equations of the
form~(\ref{horn}) is far too vast to serve as a definition of a
hypergeometric polynomial. In fact, from the point of view of the
general Definition~\ref{def:HGfunction} given above, any
polynomial in any number of variables is a hypergeometric
function. This is made precise in the following statement.

\begin{theorem}
\label{lem:anyPolyIsHG} For any polynomial
$p(x)\in\C[x_1,\ldots,x_n]$ there exists a nonconfluent~\cite{PST}
holonomic hypergeometric system of the form~(\ref{horn}) having
$p(x)$ as one of its solutions. Moreover, it can be chosen in such
a way that the hypergeometric ideal in the Weyl algebra defining
this system admits a basis that consists of a commutative family
of differential operators.
\end{theorem}
\begin{proof}
In fact, any given polynomial is annihilated by a family of
hypergeometric ideals that contains continuously many elements. To
prove the theorem, we will present an explicit representative with
desired properties.

Let $p(x)=\sum\limits_{\alpha\in S} c_{\alpha} x^{\alpha}$ be a
polynomial with the support $S.$ Here and throughout the proof we
assume $S$ to be finite. We denote by~$\# S$ the cardinality
of~$S$ and let $|\alpha|=\alpha_1 + \ldots + \alpha_n.$ For
$s\in\C^n$ define the Ore-Sato coefficient $\varphi(s)$ by
\begin{equation}
\varphi(s) = \frac{\prod\limits_{\alpha\in S}(s_1 + \ldots + s_n -
|\alpha|)}{\prod\limits_{j=1}^{n} \prod\limits_{\alpha\in S}(s_j -
\alpha_j)}. \label{OSCoeffForGenericPoly}
\end{equation}
By Definition~\ref{def:Horn system}, the action of the $j$-th
hypergeometric differential operator in the system defined by this
Ore-Sato coefficient on $p(x)$ is given by
$$
\left( x_j \prod\limits_{\alpha\in S} (\theta_1 + \ldots +
\theta_n - |\alpha|) - \prod\limits_{\alpha\in S}(\theta_j -
\alpha_j) \right) \sum\limits_{\beta\in S} c_{\beta} x^{\beta} =
$$
$$
\sum\limits_{\beta\in S} c_{\beta} \left( \left(x_j
\prod\limits_{\alpha\in S} (\theta_1 + \ldots + \theta_n -
|\alpha|)\right) x^{\beta} - \left(\prod\limits_{\alpha\in
S}(\theta_j - \alpha_j)\right) x^{\beta} \right) \equiv 0
$$
since $\C[\theta_1,\ldots,\theta_n]$ is a commutative subring in
the Weyl algebra and $(\theta_1 + \ldots + \theta_n - |\alpha|)
x^{\alpha} = (\theta_j - \alpha_j) x^{\alpha}\equiv 0$ for any
$j=1,\ldots,n.$

The hypergeometric system defined by the Ore-Sato
coefficient~(\ref{OSCoeffForGenericPoly}) is nonconfluent by
definition. By Theorem~2.8 in~\cite{Sadykov-MathScand} this system
is holonomic with the holonomic rank~$(\# S)^n.$ The fact that it
is generated by a commutative family of hypergeometric operators
follows from Lemma~2.5 in~\cite{Sadykov-MathScand}.
\end{proof}

Since every monomial in $p(x)$ is annihilated by each operator in
the hypergeometric system defined
by~(\ref{OSCoeffForGenericPoly}), the same argument works for any
Puiseux polynomial with arbitrary exponents in $\C^n.$  Although
$p(x)$ is in the kernels of the differential operators that form a
holonomic hypergeometric system, the monomials in $p(x)$ are in no
way related to each other. To obtain a meaningful definition of a
hypergeometric polynomial based on~(\ref{def:Horn system}), we
will impose further assumptions on the Ore-Sato coefficient that
defines the system.

Throughout the rest of the paper we will only consider polynomials
in $n$ variables whose Newton polytopes have nonzero
$n$-dimensional volume. For if the volume of such a Newton
polytope is zero, a suitable monomial change of variables can be
used to reduce the number of variables.

From now on we will adopt the following definition.

\begin{definition}\label{def:Z-convexSet}\rm
A set~$S\subset \Z^n$ is called $\Z^n$-{\it convex} if the
condition $\{ \lambda s^{(0)} + (1-\lambda)s^{(1)} : \lambda \in
[0,1] \}\cap \Z^n \subset S$ holds for any $s^{(0)},s^{(1)}\in S.$
\end{definition}

\begin{definition}\label{def:Z-connectedSet}\rm
A set~$S \subset \Z^n$ is said to be $\Z^n$-{\it connected} if any
two points of this set can be connected by a polygonal line with
unit sides and vertices in~$S.$
\end{definition}

For instance, the support of the bivariate hypergeometric
polynomial $x + y + 6xy + x^2y^2$ in
Example~\ref{ex:HGExampleCorrectingMikhalkin} is a $\Z^2$-convex
but {\it not} a $\Z^2$-connected set. In fact, it consists of two
$\Z^2$-connected components: $\{(1,0),(0,1),(1,1)\}$ and
$\{(2,2)\}.$

Recall that the support~$S$ of a solution to the hypergeometric
system~(\ref{horn}) is called {\it irreducible} if there is no
nonzero solution to~(\ref{horn}) supported in a proper nonempty
subset of~$S$. Any irreducible support of a solution
to~(\ref{horn}) is always a $\Z^n$-connected set. The following
statement has been established in~\cite{Sadykov-MathScand}.

\begin{lemma}\label{lem:HGpolyHasConvSupp}
(See~\cite{Sadykov-MathScand}.) If the support~$S$ of a polynomial
solution to the system~(\ref{horn}) is irreducible, then~$S$ is a
$\Z^n$-convex set.
\end{lemma}

We next show that any convex integer polytope supports an
irreducible solution to a suitable instance of the hypergeometric
system~(\ref{horn}).

\begin{lemma}\label{lem:HGpolyMightHaveAnyNewtonPolytope}
For any convex integer polytope~$P\in\R^n$ such that $P\cap\Z^n$
is $\Z^n$-connected, there exists a hypergeometric system of the
form~(\ref{horn}) and its polynomial solution~$p(x)$ with
irreducible support such that $\mathcal{N}_{p(x)}=P.$
\end{lemma}
\begin{proof}
Let $\langle B_{j}, s\rangle + c_{j} =0, \, j=1,\ldots,q$ be the
equations of the hyperplanes containing the faces of~$P$ with
$B_{j}$ being the outer normal to~$P$ at the respective face.
Since~$P$ is an integer polytope, we may without loss of
generality assume the components of the vector~$B_{j}$ to be
integer and relatively prime.

Consider the Ore-Sato coefficient
$$
\varphi(s) = \prod\limits_{j=1}^{q} \Gamma(\langle B_{j}, s\rangle
+ c_{j}).
$$
By Definition~\ref{def:Horn system} the hypergeometric system
defined by $\varphi(s)$ only depends on the quotients
$R_{j}(s):=\varphi(s+e_j)/\varphi(s)$ that are rational functions
in~$s.$ Using the $\Gamma$-function identity
$\Gamma(z)=\frac{\pi}{\Gamma(1-z) \sin(\pi z)}$ together with the
fact that the meromorphic function $\frac{e^{i\pi z}}{\sin(\pi
z)}$ is periodic with the period~$1,$ we conclude that the
quotients~$R_{j}(s)$ coincide with those for the entire function
\begin{equation}
\frac{\exp\left(i\pi\left(\sum\limits_{j=1}^{q} \langle B_{j},
s\rangle + c_{j}\right) \right)}{\prod\limits_{j=1}^{q}
\Gamma(1-\langle B_{j}, s\rangle - c_{j})}.
\label{OScoeffOfHgPoly}
\end{equation}

The support of the polynomial with the coefficient
(\ref{OScoeffOfHgPoly}) does not change if we replace the
numerator of (\ref{OScoeffOfHgPoly}) by~1. In fact, this numerator
is the exponential part $t^s$ of the Ore-Sato coefficient
(\ref{oresatocoeff}) and by
Lemma~\ref{lem:transfThatKeepAmoebaProp} affects neither the
support of the polynomial solutions to the corresponding
hypergeometric system nor the topological properties of their
amoebas. For these reasons we define the coefficient of the
hypergeometric polynomial under construction to be

\begin{equation}
\psi_{P}(s) := \frac{1}{\prod\limits_{j=1}^{q} \Gamma(1-\langle
B_{j}, s\rangle - c_{j})}. \label{realOScoeffOfHgPoly}
\end{equation}

The function $\psi_{P}(s)$ is completely defined by the integer
polytope~$P.$ By the construction the function $\psi_{P}(s)$
vanishes at any lattice point that does not belong to~$P$ and is
positive in~$P.$ Define the polynomial~$p(x)$ to be
$$
p(x) = \sum\limits_{s\in P\cap\midZ^n} \psi_{P}(s) x^s.
$$
By the explicit construction, the polynomial~$p(x)$ is supported
in~$P\cap\Z^n$ and satisfies the hypergeometric system ${\rm
Horn}(\psi_{P}(s)).$ The support is irreducible since $P\cap\Z^n$
is $\Z^n$-connected and since $\psi_{P}(s)$ does not vanish
in~$P.$
\end{proof}

The conclusion of the above lemma still holds even without the
condition of $\Z^n$-connectedness of the set of integer points in
the defining polynomial. However, the support of the polynomial
produced by the construction in the proof of the lemma will no
longer be irreducible. Such a polynomial cannot be considered as
"truly hypergeometric" since it is a linear combination of two or
more polynomials satisfying the same hypergeometric system, see
Example~\ref{ex:HGExampleCorrectingMikhalkin}. The properties of
the amoeba of such a polynomial are in general heavily dependent
on the coefficients of this linear combination. Yet, it is always
possible to choose these coefficients in such a way that the
hypergeometric polynomial with the (reducible) support $P\cap\Z^n$
is optimal. Thus we may and will without loss of generality assume
throughout the rest of the paper that the set $P\cap\Z^n$ is
$\Z^n$-connected.

\begin{remark}\label{rem:motivationForDefOfHGpoly}\rm
One can still {\it not} define a hypergeometric polynomial to be a
polynomial solution to~(\ref{horn}) with a $\Z^n$-convex
irreducible support~$S$ since any polynomial supported in~$S$ will
satisfy this condition. This can be seen by introducing more
factors into the Ore-Sato coefficient that will affect the
coefficients of the polynomial solution but will not corrupt its
support. Instead, we will distinguish the only polynomial that has
support~$S$ and satisfies the hypergeometric system of the
smallest possible holonomic rank.
\end{remark}

The following definition is central in the paper and brings
together the intrinsic properties of the classical families of
hypergeometric polynomials: the denseness of the support, the
irreducibility of the support and the property of being a solution
to a suitable system of linear differential equations with
polynomial coefficients.

\begin{definition}\label{def:trulyHGpolynomial}\rm
By a {\it multivariate hypergeometric polynomial} supported in a
convex integer polytope $P\in\R^n,$ $n\geq 2$ we will mean the
polynomial
\begin{equation}\label{myHGpolyGenForm}
\sum\limits_{s\in P\cap\midZ^n} \psi_{P}(s) x^s
\end{equation}
with $\psi_{P}(s)$ defined by~(\ref{realOScoeffOfHgPoly}).
\end{definition}

By construction, a translation of the defining polytope~$P$ by an
integer vector results in multiplication of the corresponding
hypergeometric polynomial with a monomial. Thanks to
Lemma~\ref{lem:transfThatKeepAmoebaProp} this does not affect the
amoeba of~(\ref{myHGpolyGenForm}). Throughout the rest of the
paper we will identify polytopes that are translations of each
other with respect to an integer vector.

\begin{remark}\rm
The hypergeometric polynomial introduced in
Definition~\ref{def:trulyHGpolynomial} satisfies the
hypergeometric system of the smallest possible holonomic rank
among all hypergeometric systems that admit an irreducible
polynomial solution with the support~$P.$ This property can be
used as the definition of a $P$-supported hypergeometric
polynomial. \label{rem:minimalHoloRank}
\end{remark}

In one dimension, Definition~\ref{def:trulyHGpolynomial} yields a
class of polynomials that is far too small to be interesting.
Namely, for a segment $[a,b]\subset\R$ with the integer endpoints
$a,b$ the polynomial $x^a(x+1)^{b-a}$ is hypergeometric in the
sense of Definition~\ref{def:trulyHGpolynomial}, satisfies a
hypergeometric differential equation (that is a special instance
of~(\ref{horn})) of the smallest possible holonomic rank~1 and
vanishes on an optimal algebraic set. In what follows we will
focus on the multivariate case.

\begin{example}\label{rem:excludeTrivialPolytopes}\rm Here we
compute hypergeometric polynomials associated with certain
families of integer convex polytopes and investigate their
properties.

1) {\it Direct product of segments.} If the polytope~$P$ in
Definition~\ref{def:trulyHGpolynomial} is the direct product of
segments, we may without loss of generality assume it to be
$P=[0,a_1] \times [0,a_2] \times \ldots \times [0,a_n]$ for
$a_j\in\N.$ In this case the Ore-Sato
coefficient~(\ref{realOScoeffOfHgPoly}) is given by
$$
\prod\limits_{j=1}^n
\left(\Gamma(s_j+1)\Gamma(a_j-s_j+1)\right)^{-1}.
$$
The corresponding hypergeometric polynomial is a constant multiple
of $\prod\limits_{j=1}^n (x_j+1)^{a_j}.$ Its amoeba is the union
of the coordinate hyperplanes in~$\R^n$ and is optimal in the
sense of Definition~\ref{def:optimalAmoeba}.

2) {\it A simplex.} Let now the polytope~$P\subset\R^n$ be defined
as the convex hull of the origin and the points
$(0,\ldots,k,\ldots,0)$ ($k$~in the $j$-th position) for
$j=1,\ldots,n.$ The corresponding Ore-Sato coefficient is given by
$$
\left(\Gamma\left(k+1 - \sum\limits_{j=1}^n s_j \right)
\prod\limits_{j=1}^{n} \Gamma(s_j+1) \right)^{-1}.
$$
The hypergeometric polynomial defined by this coefficient is a
constant multiple of $(x_1 + \ldots + x_n + 1)^k.$ It is optimal
in the sense of Definition~\ref{def:optimalAmoeba}, its amoeba
being just a hyperplane amoeba~\cite{FPT}.

3) {\it Cross-polytopes.} Recall that the $n$-dimensional
cross-polytope is the convex polytope whose vertices are all the
permutations of $(\pm 1,0,\ldots,0)\in\R^n.$ The only integer
point of a cross-polytope that is not its vertex is the origin.
The Ore-Sato coefficient defined by the $n$-dimensional
cross-polytope is given by
$$
\left(\prod_{\varepsilon_j = \pm 1, \, j=1,\ldots,n}
\Gamma(\varepsilon_1 s_1 + \ldots + \varepsilon_n s_n + 2)
\right)^{-1}.
$$

It might happen that an algebraic hypersurface is defined by a
polynomial which is irreducible in $\C[x_1,\ldots,x_n]$ but can be
factored in the ring of Puiseux polynomials
$\C[x_{1}^{1/d},\ldots,x_{n}^{1/d}]$ for some $d\in\N$ (see
definition of strong irreducibility below). If both factors
represent branches of the same hypersurface, such a polynomial
cannot be optimal. It turns out that a polynomial of this kind can
be hypergeometric in the sense of
Definition~\ref{def:trulyHGpolynomial}. For instance, let $n=2$
and the polytope~$P$ be the two-dimensional cross-polytope with
the vertices $(1,0),(0,1),(-1,0),(0,-1).$ The Ore-Sato coefficient
associated with this polytope is
$$
\left(\Gamma(2-s-t)\Gamma(2-s+t)\Gamma(2+s-t)\Gamma(2+s+t)\right)^{-1}
$$
and the corresponding hypergeometric polynomial is given (up to a
monomial multiple that is unimportant due to
Lemma~\ref{lem:transfThatKeepAmoebaProp}) by $p(x,y)=x + y + 4 xy
+ x^2 y + x y^2.$ This polynomial is {\it not} optimal. This can
be seen by the direct computation of its amoeba. An alternative
way to prove this is to observe that the origin must belong to the
component of order $(1,1),$ if any, in the complement of the
amoeba of the polynomial $x + y + a xy + x^2 y + x y^2.$ However,
$p(-1,-1)=0$ while the point $(-1,-1)$ is mapped to the origin by
the map~Log. In fact, $p(x,y)$ is "on the boundary" of the set of
optimal polynomials supported in~$P.$

The above becomes clear in view of the following factorization of
$p(x,y)$ in $\C[\sqrt{x},\sqrt{y}]:$
$p(x,y)=\left(\sqrt{x}(1+y)+\sqrt{-1}\sqrt{y}(1+x)\right)\left(\sqrt{x}(1+y)-\sqrt{-1}\sqrt{y}(1+x)\right).$
Both factors represent branches of the same algebraic hypersurface
$\{p(x,y)=0\}$ which can therefore {\it not} be optimal.

We remark that the generic polynomial $p(x_1,\ldots,x_n)= c
+\sum\limits_{j=1}^{n}\left(a_j x_j + \frac{b_j}{x_j}\right)$
supported in the $n$-dimensional cross-polytope and having
positive coefficients $a_j,b_j,c$ is optimal if and only if
$\sum\limits_{j=1}^{n}\sqrt{a_j b_j} < \frac{c}{2}.$ Thus the
hypergeometric polynomial defined by the $n$-dimensional
cross-polytope is optimal if and only if $n>2.$

4) {\it The Hirzebruch surface.} Recall that the Hirzebruch
surface~$\mathbb{F}_1$ is defined by the fan generated by $(1,0),$
$(0,1),$ $(-1,1)$ and $(0,-1).$ The hypergeometric polynomial
supported in (a translation of) the convex hull of these vectors
is defined by the Ore-Sato coefficient
$$
\left(\Gamma(3-t)\Gamma(4-s-t)\Gamma(t
-s+2)\Gamma(2s+t-1)\right)^{-1}.
$$
This polynomial is a constant multiple of $3 x+12 x y+2 x^2 y+2
y^2+3 x y^2$ and is optimal.

\begin{figure}[ht]
\begin{minipage}[h]{0.24\linewidth}
\center{\includegraphics[width=0.99\linewidth]{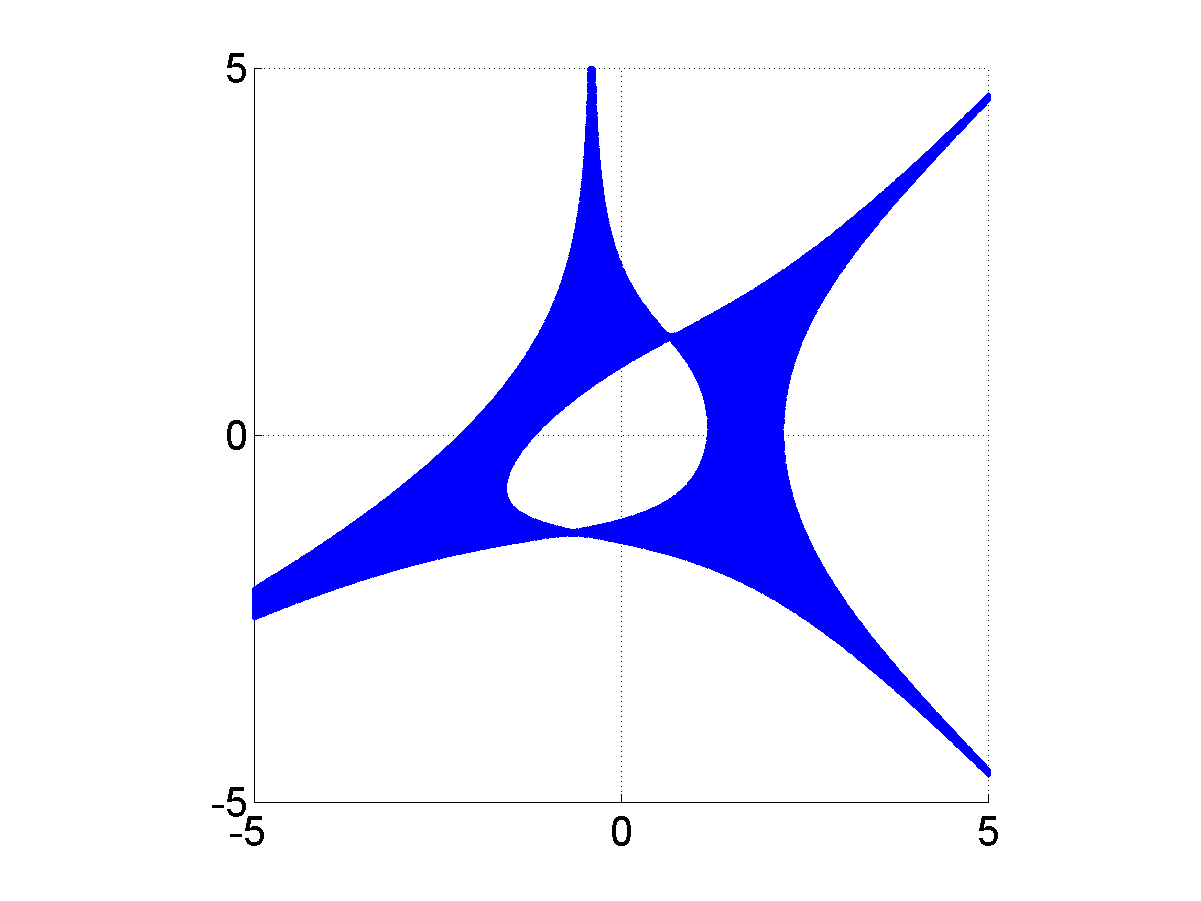} \\ (a)}
\end{minipage}
\begin{minipage}[h]{0.24\linewidth}
\center{\includegraphics[width=0.99\linewidth]{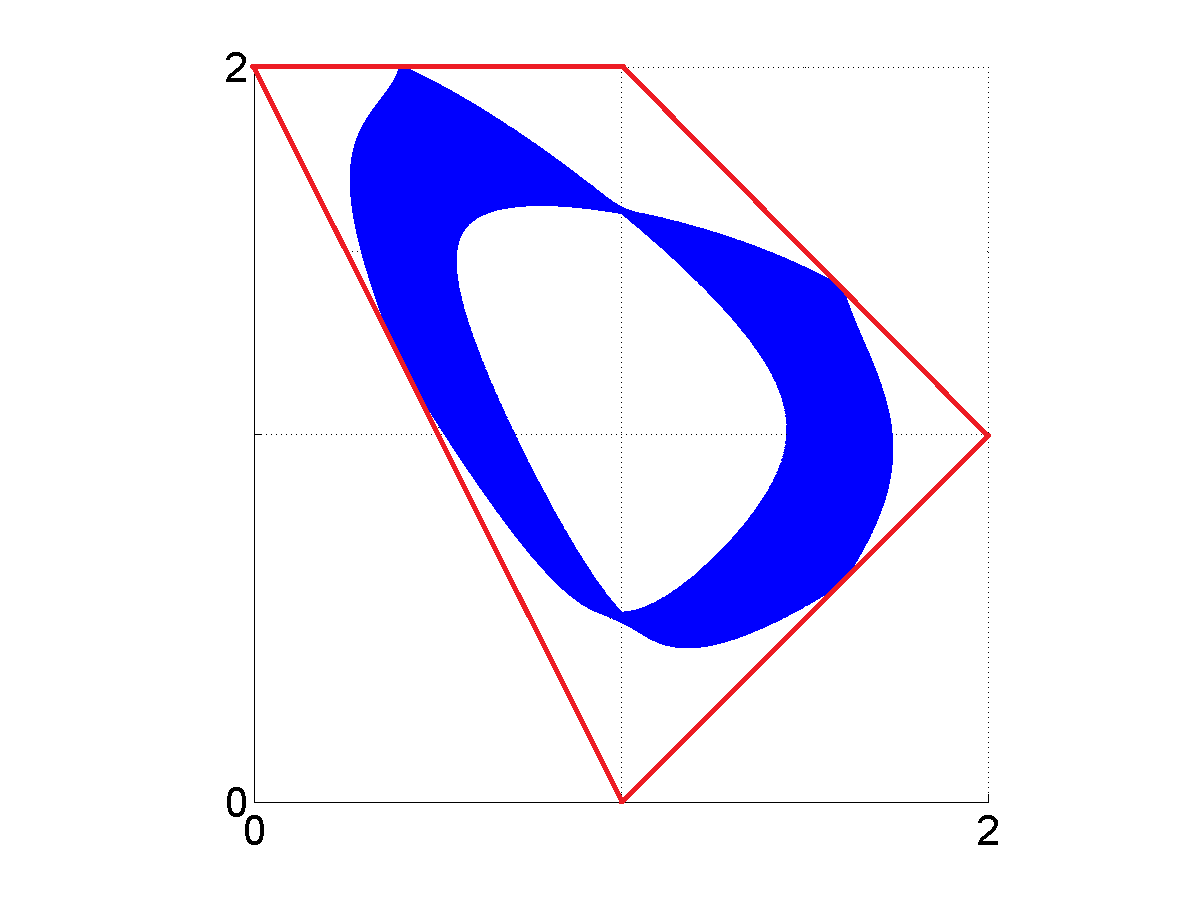} \\ (b)}
\end{minipage}
\begin{minipage}[h]{0.24\linewidth}
\center{\includegraphics[width=0.99\linewidth]{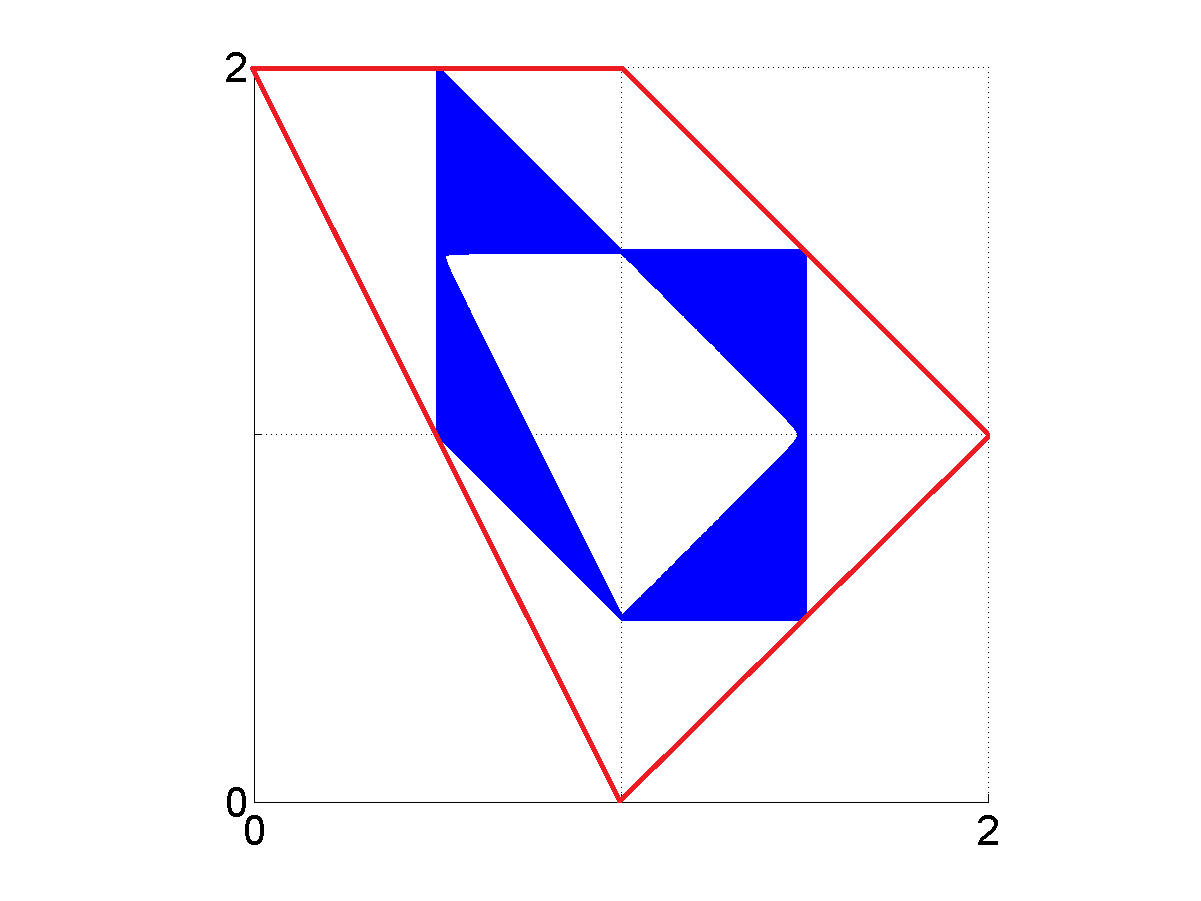} \\ (c)}
\end{minipage}
\begin{minipage}[h]{0.24\linewidth}
\center{\includegraphics[width=0.99\linewidth]{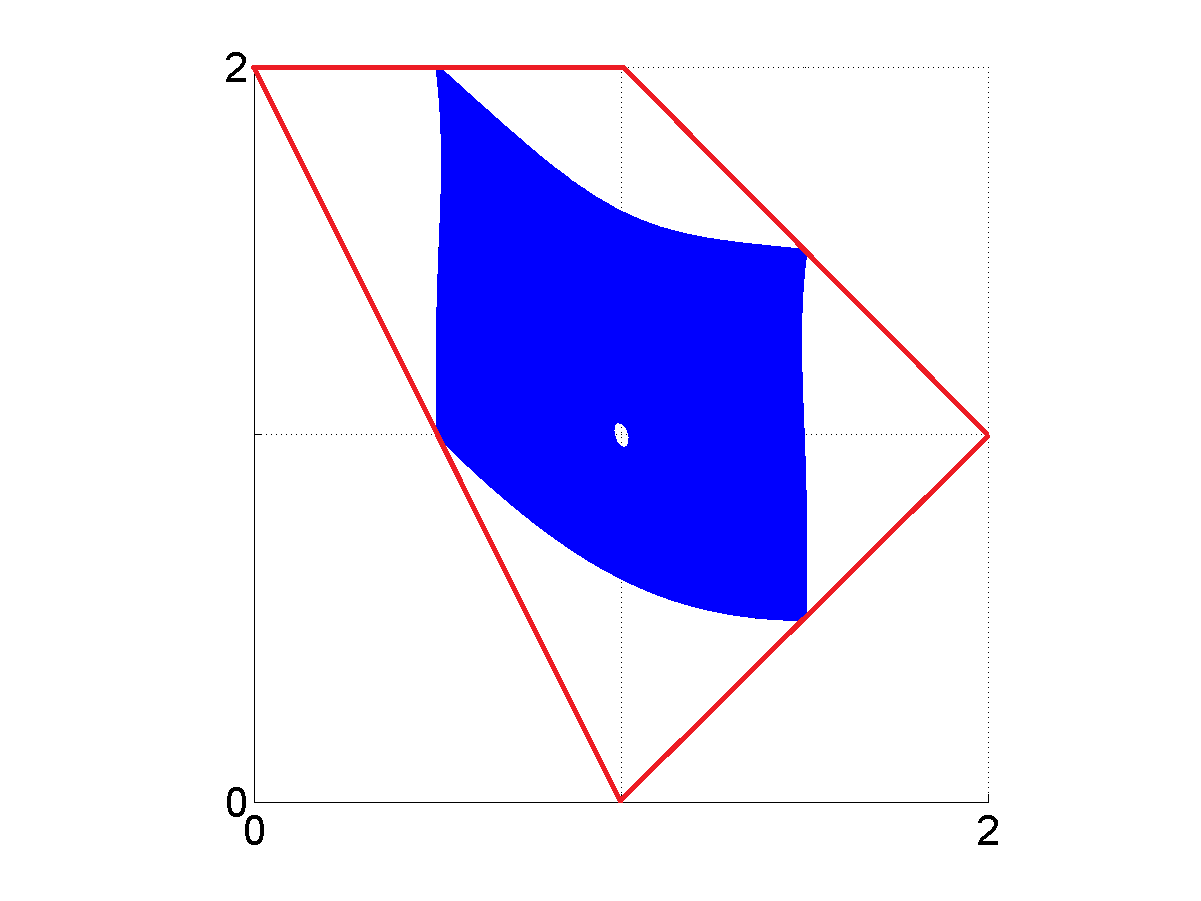} \\ (d)}
\end{minipage}
\caption{The amoeba, compactified amoeba, the weighted
compactified amoeba of the 6th Hadamard power and a vanishing
connected component of the complement to the
$\mathcal{WCA}(p_{3}(x,y))$ (see Definition~\ref{def:weightedCompactifiedAmoeba})} \label{fig:Hirzebruch}
\end{figure}

\end{example}

Recall that a (Laurent) polynomial
$p(x_1,\ldots,x_n)\in\C[x_{1}^{\pm 1},\ldots,x_{n}^{\pm 1}]$ is
called {\it strongly irreducible} if for any nondegenerate integer
matrix~$v$ the polynomial $p(x^{v_1},\ldots,x^{v_n})$ is
irreducible over~$\C$. Here $v_j$ is the $j$th row of the
matrix~$v.$ This definition is only meaningful in $n\geq 2$
variables since there are no univariate strongly irreducible
polynomials. The property of a polynomial being strongly
irreducible in $\C[x_{1}^{\pm 1},\ldots,x_{n}^{\pm 1}]$ is a
generic condition in the Zariski topology.
Example~\ref{rem:excludeTrivialPolytopes}~3) provides a polynomial
that is irreducible but is not strongly irreducible.

By Lemma~\ref{lem:transfThatKeepAmoebaProp} the amoeba of a
polynomial that is not strongly irreducible is the union of the
amoebas of its Puiseux polynomial factors. We thus may without any
loss of generality restrict our attention to strongly irreducible
hypergeometric polynomials.

The hypergeometric polynomials introduced in
Definition~\ref{def:trulyHGpolynomial} enjoy properties that are
parallel to the properties 1--6 listed in the Introduction of the
classical hypergeometric polynomials. Indeed, for any integer
convex polytope there is only one (up to a constant multiple)
hypergeometric polynomial of the form~(\ref{myHGpolyGenForm})
supported in this polytope; this polynomial is dense; it satisfies
the hypergeometric system ${\rm Horn}(\psi_{P}(s)).$ The
counterpart of the properties of the roots
of~(\ref{myHGpolyGenForm}) is given by the next theorem which is
the main result of the paper.

\begin{theorem}\label{thm:HGpolyHasOptimalAmoeba}
Strongly irreducible hypergeometric polynomials are optimal.
\end{theorem}
\begin{proof}
Let $p(x)$ be a hypergeometric polynomial in the sense of
Definition~\ref{def:trulyHGpolynomial} with the variables
$x=(x_1,\ldots,x_n)\in\C^n$ and supported in a finite
$\Z^n$-connected set $S\subset\Z^n.$ Denote by $\varphi(s)$ the
Ore-Sato coefficient of $p(x).$ By definition, the function
$\varphi(s)$ is well-defined, finite and positive on~$S.$

By the comment after Remark~\ref{rem:minimalHoloRank} a univariate
polynomial that is hypergeometric in the sense of
Definition~\ref{def:trulyHGpolynomial} is optimal. From now on we
will consider the case $n\geq 2.$

By~\cite{FPT}, the order of the connected component~$M$ of the
complement of the amoeba $\mathcal{A}_f$ is given by the vector
with the coordinates
$$
\frac{1}{(2\pi i)^n}\int\limits_{\rm{Log}^{-1}(\xi)}
\frac{\theta_{j}f(x)}{f(x)} \frac{dx_1 \wedge\ldots\wedge
dx_n}{x_1\ldots x_n}, \quad j = 1,\ldots,n.
$$
Here $\xi$ is any point in $M.$ We will show that for any
$\alpha\in S$ there exists a connected component of order
$\alpha\in S\subset\Z^n$ in the amoeba
complement~$^c\!\mathcal{A}_{p(x)}.$ By
Lemma~\ref{lem:transfThatKeepAmoebaProp} it suffices to consider
the case when~$\alpha$ belongs to the interior of the convex hull
(in~$\R^n$) of~$S.$ Indeed, if~$\alpha$ is a vertex of
$\mathcal{N}_{p(x)}$ then the desired conclusion is established in
Theorem~\ref{3thmfptestimate}. If~$\alpha$ is on the boundary of
the convex hull of~$S$ but is not its vertex, then a suitable
nondegenerate monomial change of variables can be used to reduce
the dimension of the variable space and put~$\alpha$ into the
interior of the Newton polytope of a polynomial in fewer variables
obtained from $p(x)$ by setting some of its variables to zero.

Furthermore, we may without loss of generality assume that
$\alpha=0.$ Indeed, by Lemma~\ref{lem:transfThatKeepAmoebaProp}
the amoebas of $p(x)$ and $x^{-\alpha} p(x)$ are the same and the
component of order~$\alpha$ in $^c\!\mathcal{A}_{p(x)}$ coincides
with the component of order~$0$ in
$^c\!\mathcal{A}_{x^{-\alpha}p(x)}.$

By the Bohr-Mollerup theorem the
coefficient~(\ref{realOScoeffOfHgPoly}) of a hypergeometric
polynomial is a positive and strictly logarithmically concave
function on the convex hull of the support of this polynomial. It
follows from~\cite[Corollary~2]{Rullgaard} that a sufficiently
high positive (integer) power~$\ell$ of the Ore-Sato
coefficient~$\varphi(s)$ defines an optimal hypergeometric
polynomial. In fact, any sufficiently big positive real power
of~$\varphi(s)$ will define an optimal polynomial but it will only
satisfy a hypergeometric system of equations for integer powers.

It remains to check that $\ell=1$ works. Let the Ore-Sato
coefficient~$\psi_{P}(s)$ be defined
by~(\ref{realOScoeffOfHgPoly}). By definition the function
$$
\tilde{p}(x):=\sum\limits_{\tiny\begin{array}{c}s\in S \\
s\neq 0 \end{array}} \psi_{P}(s) x^s
$$
is positive and convex in~$\R_{+}^{n}$. It tends to infinity as
$x_{j}\rightarrow 0$ or $x_{j}\rightarrow \infty$ for any
$j=1,\ldots,n.$ Thus there exists the unique minimum
$\rho=(\rho_1,\ldots,\rho_n)\in\R^n$ such that
$\tilde{p}(x)\geq\tilde{p}(\rho)$ for any $x\in\R_{+}^{n}.$

\begin{definition} \rm Following the
ideas of~\cite{Zharkov}, we define the {\it weighted moment map}
associated with the algebraic hypersurface $\{ x\in\C^n : f(x):=
\sum\limits_{s\in S} a_s x^s = 0\}$ through
$$
\mu_{f}(x) := \frac{\sum\limits_{s\in S} s\cdot |a_s|
|x^s|}{\sum\limits_{s\in S} |a_s| |x^s|}.
$$
\label{def:weightedMomentMap}
\end{definition}
It follows from the general theory of moment maps~\cite{Guillemin}
that $\mu_{f}(\C^n)\subseteq \mathcal{N}_{f}.$
\begin{definition}\rm
By the {\it weighted compactified amoeba} of an algebraic
hypersurface $H=\{x\in\subset\C^n : f(x)=0\}$ we will mean the set
$\mu_{f}(H).$ We denote it by $\mathcal{WCA}(f).$
\label{def:weightedCompactifiedAmoeba}
\end{definition}
By~\cite{Guillemin} there exists a deformation~$\hat{p}(x)$ of the
Laurent polynomial~$\tilde{p}(x)$ such that
$\tilde{p}(x)-\hat{p}(x)=\delta\in\R_{+}$ and moreover the
complement of $\mathcal{WCA}(\hat{p}(x)-\varepsilon)$ contains the
connected component~$M\ni 0\in\R^n$ of order $0\in\Z^n$ for any
$\varepsilon\geq 0.$ In other words, there exists a deformation of
the constant term of~$\tilde{p}(x)$ which makes the complement
component of order zero vanish at the origin. This is illustrated
by Fig.~\ref{fig:Hirzebruch}~{\it(d)}
and~\ref{fig:quadrilateral}~{\it(d)}.

It is therefore sufficient to check that the origin cannot belong
to $\mathcal{WCA}(\tilde{p}(x)).$ The preimage of the origin under
the map~$\mu_{\tilde{p}}$ is given by
$$
\Pi=\left\{ x\in\C^n :
\frac{\partial\tilde{p}(\zeta)}{\partial\zeta_j}\Big|_{\zeta_1=|x_1|,\ldots,\zeta_n=|x_n|}
= 0 \text{\ for\ any\ } j=1,\ldots,n \right\}.
$$
Since the only point in~$\R_{+}^{n}$ where the gradient
of~$\tilde{p}(x)$ vanishes is~$\rho,$ it follows, that
$$
\Pi=\left\{ (\rho_{1} e^{i\chi_1},\ldots,\rho_{n} e^{i\chi_n}) :
\chi_j\in [0,2\pi] \right\} .
$$
To prove the existence of the complement component of order zero
we need to show that the Laurent polynomial~$\tilde{p}(x)$ cannot
vanish on~$\Pi.$

Consider the set~$\mathcal{L}$ of lines in~$\R^n$ such that the
intersection of any element in~$\mathcal{L}$ with~$S$ contains
three points. If~$\mathcal{L}$ is empty then the
polynomial~$\tilde{p}(x)$ is optimal by~\cite{FPT}. From now on we
assume that~$\mathcal{L}$ is nonempty. Let $L\in\mathcal{L}$ and
consider the restricted polynomial $p|_{L}(x):=\sum\limits_{s\in
S\cap L} \varphi(s) x^s.$ Since the coefficients of~$p(x)$ are
defined as the restriction of a logarithmically concave function
to the support of~$p(x),$ the same is true for the coefficients of
the restricted polynomial~$p|_{L}(x).$ Making, if necessary, a
monomial change of variables, we conclude that~$p|_{L}(x) = a_L
m(x)^2 + b_L m(x) + c_L,$ where $a_L,b_L,c_L\in\Z$ and $m(x)$ is a
monomial function in the coordinates of~$x.$ Moreover, this
second-order polynomial has a logarithmically concave
hypergeometric coefficient and vanishes on a certain circle
centered at the origin.

The logarithmic concavity of the coefficient of $p|_{L}(x)$ is
equivalent to $b_L\geq\sqrt{a_L c_L}.$ Since the coefficients of
$p|_{L}(x)$ are positive integers, its roots are on the same
circle centered at the origin if and only if $p|_{L}(x) = a_L
m(x)^2 - 2 a_L m(\rho) \cos\omega\, m(x) + a_{L} m(\rho)^2,$ where
$\cos\omega\leq\frac{1}{2}.$ Such a polynomial can only be a
restriction of a multivariate hypergeometric polynomial in the
case when $\omega=\pi$ and its roots are the same. If this holds
for any $L\in\mathcal{L}$ then the initial polynomial $p(x)$
cannot be strongly irreducible.

This means that there is a connected component of order~$\alpha$
in the complement of the amoeba~$\mathcal{A}_{p(x)}.$ Since
$\alpha\in S\subset\Z^n$ was arbitrary, the amoeba is optimal.
\end{proof}

We stress once again that by
Definition~\ref{def:trulyHGpolynomial} the optimality of a
polynomial is the property of its zero locus and not the
polynomial itself. For instance, the bivariate hypergeometric
polynomial $(1-x-y)^2$ is optimal since its zeros form the optimal
algebraic hypersurface $\{x+y=1\}.$

The class of dense optimal multivariate polynomials with
$\Z^n$-convex supports is of course much wider than the class of
hypergeometric polynomials. By~\cite{Rullgaard} the coefficient of
a polynomial only has to be "logarithmically concave enough" for
the polynomial itself to be optimal.

Example~\ref{rem:excludeTrivialPolytopes}~{\it(a)} shows that the
condition of (strong) irreducibility is in general not necessary
for a hypergeometric polynomial to be optimal. On the other hand,
the polynomial in
Example~\ref{rem:excludeTrivialPolytopes}~{\it(c)} fails the
strong irreducibility condition and is not optimal due to the fact
that it factors into the product of two different Puiseux
polynomials with the same amoebas. Throughout the rest of the
paper we will without loss of generality only consider strongly
irreducible hypergeometric polynomials.

\begin{example}\label{ex:(Gamma[t-1]Gamma[4s-2t+1]Gamma[-4s-4t+25])^(-1)}\rm
The bivariate Ore-Sato coefficient
$$
\varphi(s,t) =
\left(\Gamma(t+1)\Gamma(1+6s-3t)\Gamma(31-6s-2t)\right)^{-1}
$$
defines a confluent holonomic hypergeometric system with the
polynomial solution
$$
\begin{array}{c}
p_{0}(x,y) = 1+593775 x+86493225 x^2+86493225 x^3+593775
x^4+x^5+39331656000 x y + \\
34936343442000 x^2 y+55898149507200 x^3 y+216324108000 x^4
y+54513675216000 x y^2 + \\
2112950051372160000 x^2 y^2+6867087666959520000 x^3
y^2+10357598291040000 x^4 y^2+ \\
15382276373989324800000 x^2 y^3+169205040113882572800000 x^3 y^3+\\
33807200821954560000 x^4 y^3+3045690722049886310400000 x^2 y^4+\\
639595051630476125184000000 x^3 y^4+184203374869577124052992000000
x^3 y^5+ \\
368406749739154248105984000000 x^3 y^6.
\end{array}
$$
The support of this hypergeometric polynomial is $\Z^2$-convex and
has a triangular convex hull. The amoeba of $p_{0}(x,y)$ is
optimal, see
Fig.~\ref{fig:(Gamma[t-1]Gamma[4s-2t+1]Gamma[-4s-4t+25])^(-1)}.
\end{example}

\begin{figure}[ht!]
\begin{picture}(120,120)

\put(0,0){\circle*{5}} \put(20,0){\circle*{5}}
\put(40,0){\circle*{5}} \put(60,0){\circle*{5}}
\put(80,0){\circle*{5}} \put(100,0){\circle*{5}}

\put(20,20){\circle*{5}} \put(40,20){\circle*{5}}
\put(60,20){\circle*{5}} \put(80,20){\circle*{5}}

\put(20,40){\circle*{5}} \put(40,40){\circle*{5}}
\put(60,40){\circle*{5}} \put(80,40){\circle*{5}}

\put(40,60){\circle*{5}} \put(60,60){\circle*{5}}
\put(80,60){\circle*{5}}

\put(40,80){\circle*{5}} \put(60,80){\circle*{5}}

\put(60,100){\circle*{5}}

\put(60,120){\circle*{5}}

\put(0,0){\line(1,0){100}} \put(0,0){\line(1,2){60}}
\put(60,120){\line(1,-3){40}}

\end{picture} \quad\quad\quad
\includegraphics[width=10cm]{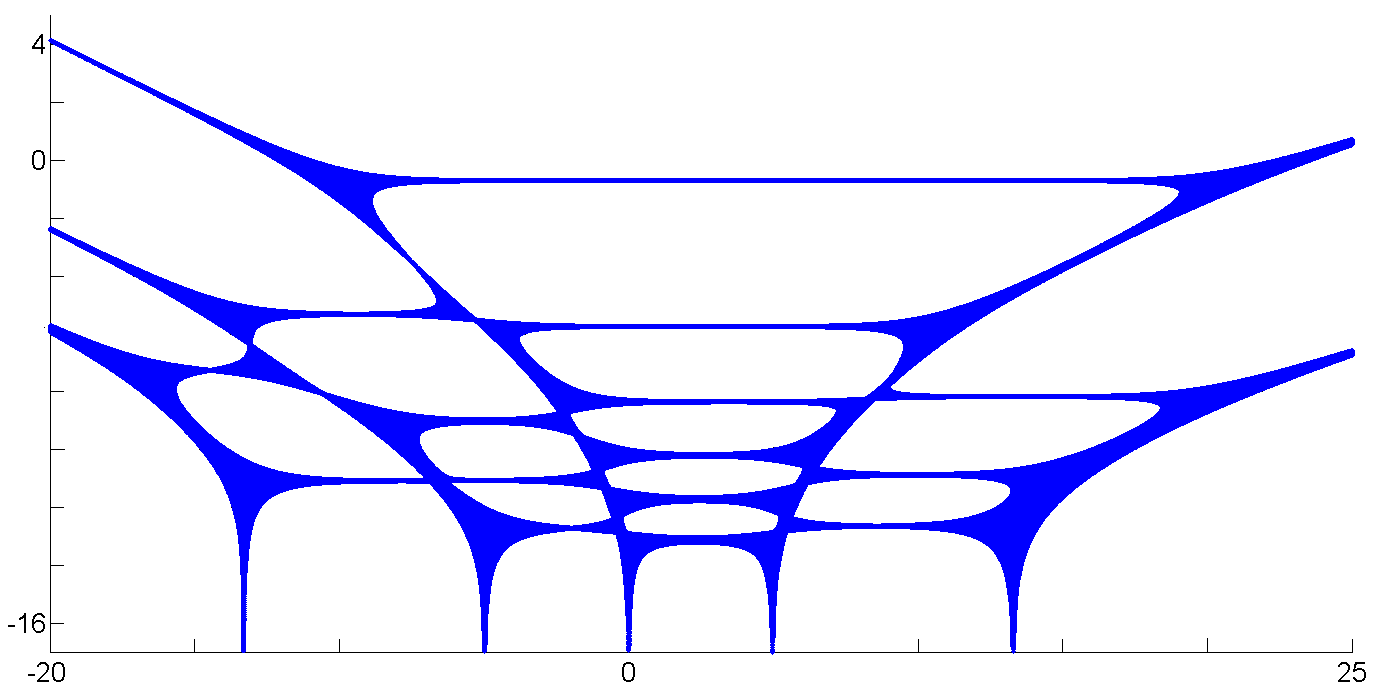}\caption{The Newton
polygon and the amoeba of $p_{0}(x,y)$}
\label{fig:(Gamma[t-1]Gamma[4s-2t+1]Gamma[-4s-4t+25])^(-1)}
\end{figure}
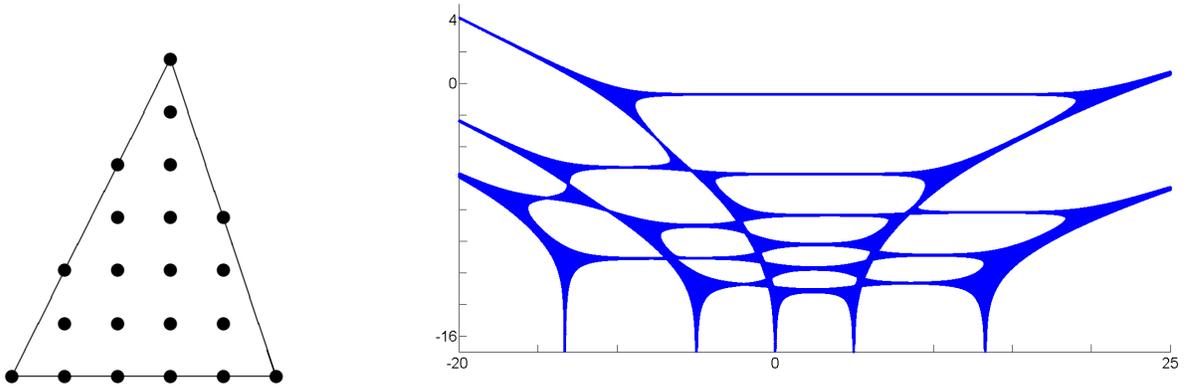

\begin{remark}\label{rem:numeratorsOfBergmanKernelsAreNotOpt}\rm
Recall that the Bergman kernel of a complex ellipsoidal domain is
given by a rational hypergeometric function~\cite{PST}. One can
check that the numerators of such rational functions are not
necessarily optimal polynomials. The amoebas of the singular
divisors of the GKZ-hypergeometric functions~\cite{Beukers2009}
are known to be solid~\cite{PST}. Thus the optimal property of the
divisors of hypergeometric polynomials cannot be extended to the
classes of rational or algebraic hypergeometric functions.
\end{remark}

Recall that the {\it Hadamard power} of order~$r\in\R$ of a
polynomial $f(x)=\sum\limits_{s\in S} a_s x^s$ is defined to be
$f^{[r]}(x):=\sum\limits_{s\in S} a_{s}^{r} x^s.$ We observe that
the set-theoretical limit
$\mathcal{S}(f):=\lim\limits_{r\rightarrow\infty}
\mathcal{WCA}(f^{[r]})\subset\mathcal{N}_{f}$ is an amoeba-like
simplicial complex. This simplicial complex for the Hirzebruch
polynomial is depicted in Fig.~\ref{fig:Hirzebruch}~{\it(c)}
inside the Newton polygon of that polynomial. An approximation of
the simplicial complex~$\mathcal{S}(p_{3}(x,y))$ is depicted in
Fig.~\ref{fig:quadrilateral}~{\it(c)}. The geometry
of~$\mathcal{S}(f)$ is related to the amoeba of~$f$ while the
combinatorics of~$\mathcal{S}(f)$ reflects intrinsic algebraic
properties of this polynomial.

\section{Classical bivariate hypergeometric polynomials}\label{sec:classicalBivariatePolys}

Despite varying terminology, the classical hypergeometric series
$F_1,\ldots,F_4,$ $G_1,\ldots,G_3,$ $H_1,\ldots,H_7$ as well as
other entries of the Horn list~\cite{Erdelyi} are universally
considered to be intrinsically hypergeometric. For resonant
parameters~\cite{SadykovTanabe}, many of these series terminate
and turn out to be bivariate hypergeometric polynomials.

Appell's~$F_1$ is one of the most important classical
hypergeometric series since by the results of~\cite{Erdelyi} any
bivariate hypergeometric system of second-order equations and
holonomic rank~3 can be transformed into the system for~$F_1$ or a
particular limiting case of this system. The following statement
follows from Theorem~\ref{thm:HGpolyHasOptimalAmoeba}.

\begin{corollary}
The polynomial instances of the Appell $F_1(a,b_1,b_2,c;x,y)$
hypergeometric function are optimal for $a,b_1,b_2,-c<0$ and
$a>b_1+b_2.$
\end{corollary}

\begin{proof}
The imposed conditions on the parameters of $F_1(a,b_1,b_2,c;x,y)$
yield a one-to-one correspondence between the $\Gamma$-factors in
the coefficient of the power series expansion of~$F_1$ and the
sides of the Newton polygon of its polynomial instance in
question. This polynomial is therefore hypergeometric in the sense
of Definition~\ref{def:trulyHGpolynomial}.
\end{proof}

In Fig.~\ref{fig:F_1(-5,-4,-4,3)} we depict the amoeba of the optimal hypergeometric polynomial $F_1(-5,-4,-4,$ $3;x,y).$

\begin{figure}[ht!]
\begin{picture}(80,80)

\put(0,0){\circle*{5}} \put(20,0){\circle*{5}}
\put(40,0){\circle*{5}} \put(60,0){\circle*{5}}
\put(80,0){\circle*{5}}

\put(0,20){\circle*{5}} \put(20,20){\circle*{5}}
\put(40,20){\circle*{5}} \put(60,20){\circle*{5}}
\put(80,20){\circle*{5}}

\put(0,40){\circle*{5}} \put(20,40){\circle*{5}}
\put(40,40){\circle*{5}} \put(60,40){\circle*{5}}

\put(0,60){\circle*{5}} \put(20,60){\circle*{5}}
\put(40,60){\circle*{5}}

\put(0,80){\circle*{5}} \put(20,80){\circle*{5}}

\put(0,0){\line(1,0){80}} \put(0,0){\line(0,1){80}}
\put(0,80){\line(1,0){20}} \put(80,0){\line(0,1){20}}
\put(20,80){\line(1,-1){60}}

\end{picture}
\quad\quad\quad
\includegraphics[width=5cm]{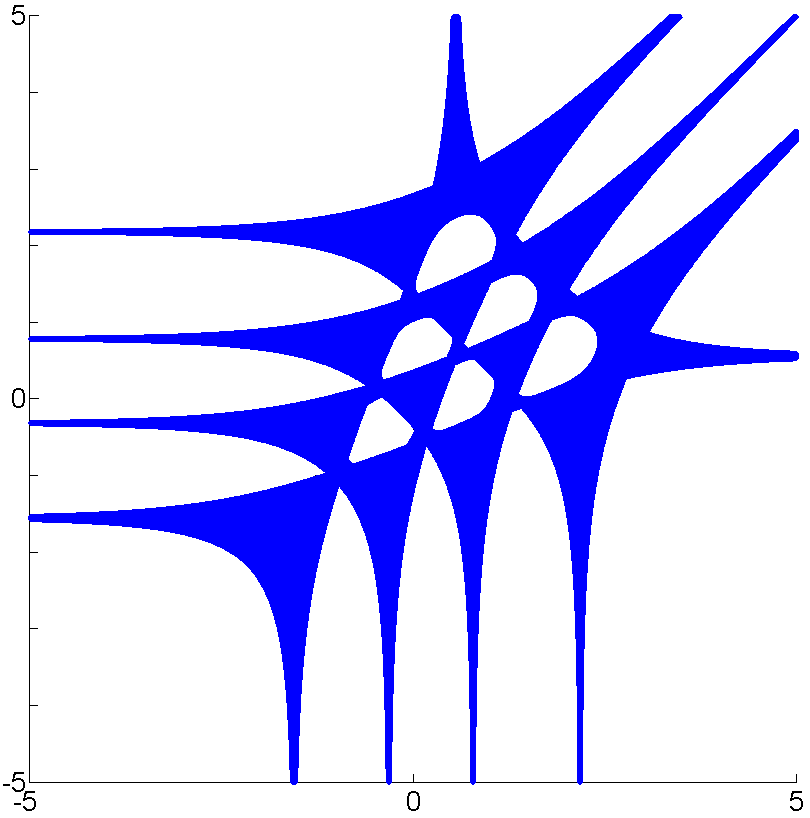}
\caption{The Newton polygon and the amoeba of the Appell
polynomial $F_1(-5,-4,-4,3;x,y)$} \label{fig:F_1(-5,-4,-4,3)}
\end{figure}
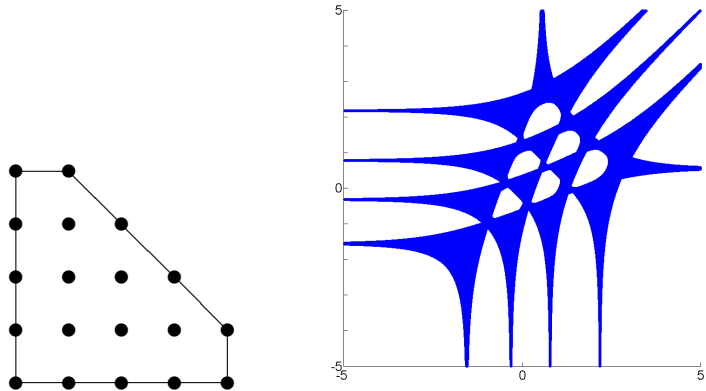

Observe however that not every polynomial instance of
$F_{1}(a,b_1,b_2,c;x,y)$ is optimal. The optimal property is in
general not possessed by the~$F_1$ polynomials whose Newton
polytopes do not have sides that are orthogonal to the gradients
of the linear forms in the defining Ore-Sato coefficient. For
instance, $F_{1}(-4,5,-7,9;x,y)$ is {\it not} an optimal
polynomial, its Newton polygon being just a triangle.

We further remark that the zero locus of a rational instance of a
classical hypergeometric function need not be an optimal
hypersurface. For example, the numerator of the rational function
$F_2(5;3/2,1;-1/2,2;x,y)$ is not an optimal polynomial.

\section{examples}\label{sec:examples}

In this section we collect examples of multivariate hypergeometric
polynomials together with their Newton polytopes and amoebas.

\begin{example}\label{ex:8-gon}\rm
The hypergeometric Horn system defined by the Ore-Sato coefficient
$$
\varphi(s,t)=\Gamma(s-6) \Gamma(s+t-10) \Gamma(t-6)
\Gamma(-s+t-4)\Gamma(-s)\Gamma(-s-t+2) \Gamma(-t)\Gamma(s-t-4)
$$
admits the following polynomial solution:

$p_{1}(x,y)=21x^{2}+64x^{3}+21x^{4}+126xy+2016x^{2}y+4704x^{3}y+2016x^{4}y+126x^{5}y+21y^{2}+2016xy^{2}+
22050x^{2}y^{2}+47040x^{3}y^{2}+22050x^{4}y^{2}+2016x^{5}y^{2}+21x^{6}y^{2}+64y^{3}+4704xy^{3}+47040x^{2}y^{3}+98000x^{3}y^{3}+
47040x^{4}y^{3}+4704x^{5}y^{3}+64x^{6}y^{3}+21y^{4}+2016xy^{4}+22050x^{2}y^{4}+47040x^{3}y^{4}+22050x^{4}y^{4}+2016x^{5}y^{4}+
21x^{6}y^{4}+126xy^{5}+2016x^{2}y^{5}+4704x^{3}y^{5}+2016x^{4}y^{5}+126x^{5}y^{5}+21x^{2}y^{6}+64x^{3}y^{6}+21x^{4}y^{6}.$

(The system itself is too cumbersome to display and we omit it.)
The Newton polygon and the amoeba of $p_{1}(x,y)$ are shown in
Fig.~\ref{fig:8-gonAndAmoeba}. This amoeba turns out to be
optimal.
\end{example}

\begin{figure}[ht!]
\begin{picture}(120,120)

\put(40,0){\circle*{5}} \put(60,0){\circle*{5}}
\put(80,0){\circle*{5}}

\put(20,20){\circle*{5}} \put(40,20){\circle*{5}}
\put(60,20){\circle*{5}} \put(80,20){\circle*{5}}
\put(100,20){\circle*{5}}

\put(0,40){\circle*{5}} \put(20,40){\circle*{5}}
\put(40,40){\circle*{5}} \put(60,40){\circle*{5}}
\put(80,40){\circle*{5}} \put(100,40){\circle*{5}}
\put(120,40){\circle*{5}}

\put(0,60){\circle*{5}} \put(20,60){\circle*{5}}
\put(40,60){\circle*{5}} \put(60,60){\circle*{5}}
\put(80,60){\circle*{5}} \put(100,60){\circle*{5}}
\put(120,60){\circle*{5}}

\put(0,80){\circle*{5}} \put(20,80){\circle*{5}}
\put(40,80){\circle*{5}} \put(60,80){\circle*{5}}
\put(80,80){\circle*{5}} \put(100,80){\circle*{5}}
\put(120,80){\circle*{5}}

\put(20,100){\circle*{5}} \put(40,100){\circle*{5}}
\put(60,100){\circle*{5}} \put(80,100){\circle*{5}}
\put(100,100){\circle*{5}}

\put(40,120){\circle*{5}} \put(60,120){\circle*{5}}
\put(80,120){\circle*{5}}

\put(40,0){\line(1,0){40}} \put(80,0){\line(1,1){40}}
\put(120,40){\line(0,1){40}} \put(120,80){\line(-1,1){40}}
\put(80,120){\line(-1,0){40}} \put(40,120){\line(-1,-1){40}}
\put(0,80){\line(0,-1){40}} \put(0,40){\line(1,-1){40}}

\end{picture} \quad\quad\quad\quad \includegraphics[width=7cm]{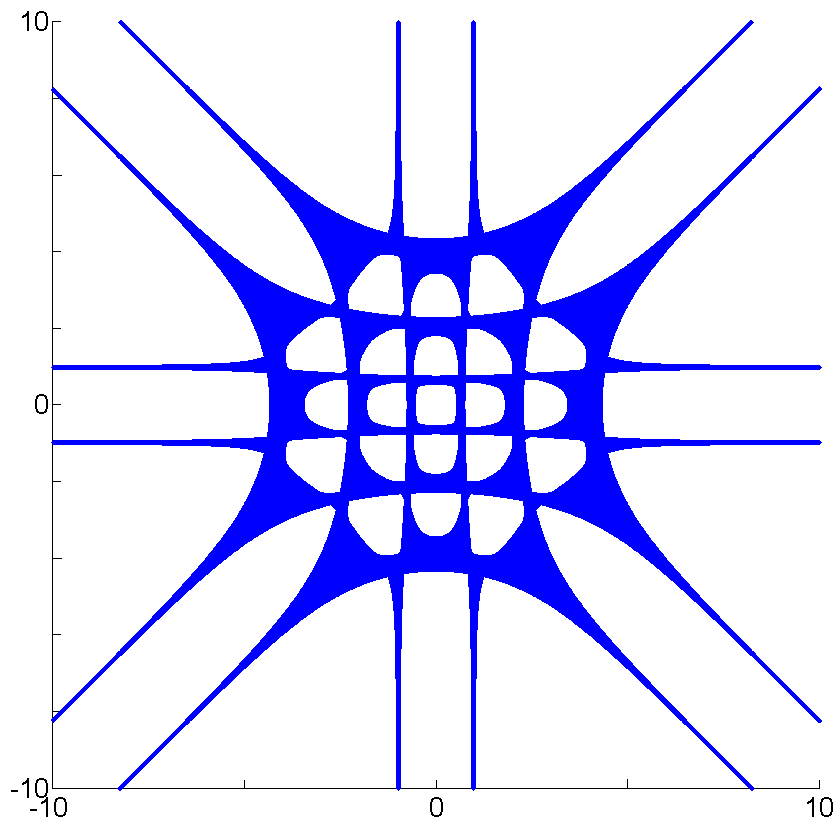}\caption{The
Newton polygon and the amoeba of $p_{1}(x,y)$}
\label{fig:8-gonAndAmoeba}
\end{figure}
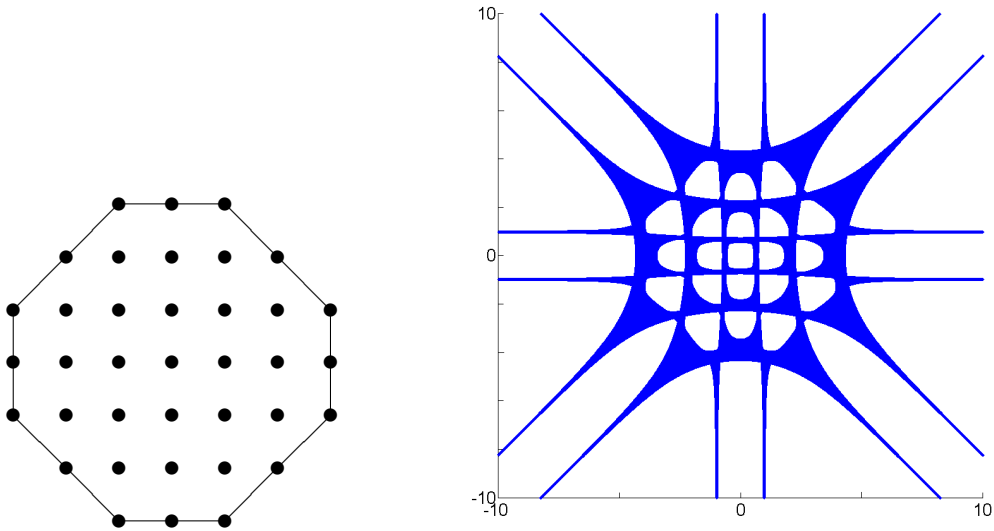

\begin{example}\label{ex:simplicialPoly}
\rm The next example shows that the number of $\Gamma$-factors in
the Ore-Sato coefficient of an optimal hypergeometric polynomial
can be strictly smaller than the number of faces of its Newton
polytope.  The hypergeometric system defined by the Ore-Sato
coefficient
\begin{equation}
\varphi(s,t)=\Gamma\left(s+2t-5\right)\Gamma\left(-2s-t-4\right)\Gamma\left(-s-5t+1\right)
\label{moreInvolvedOSCoeff}
\end{equation}
has the following polynomial solution:

$p_{2}(x,y)=
2421619200x^{5}+172972800x^{6}+2882880x^{7}+14560x^{8}+20x^{9}+174356582400x^{2}y+
48432384000x^{3}y+2421619200x^{4}y+34594560x^{5}y+160160x^{6}y+208x^{7}y+2421619200xy^{2}+
691891200x^{2}y^{2}+21621600x^{3}y^{2}+160160x^{4}y^{2}+286x^{5}y^{2}+524160xy^{3}+14560x^{2}y^{3}+
56x^{3}y^{3}+32y^{4}+xy^{4}. $

The support of $p_{2}(x,y)$ (bounded by the singular divisors of
the corresponding Ore-Sato coefficient) and its amoeba are
depicted in Fig.~\ref{fig:simplicialPolyAndAmoeba}.
\end{example}

\begin{figure}[ht!]
\includegraphics[width=8.2cm]{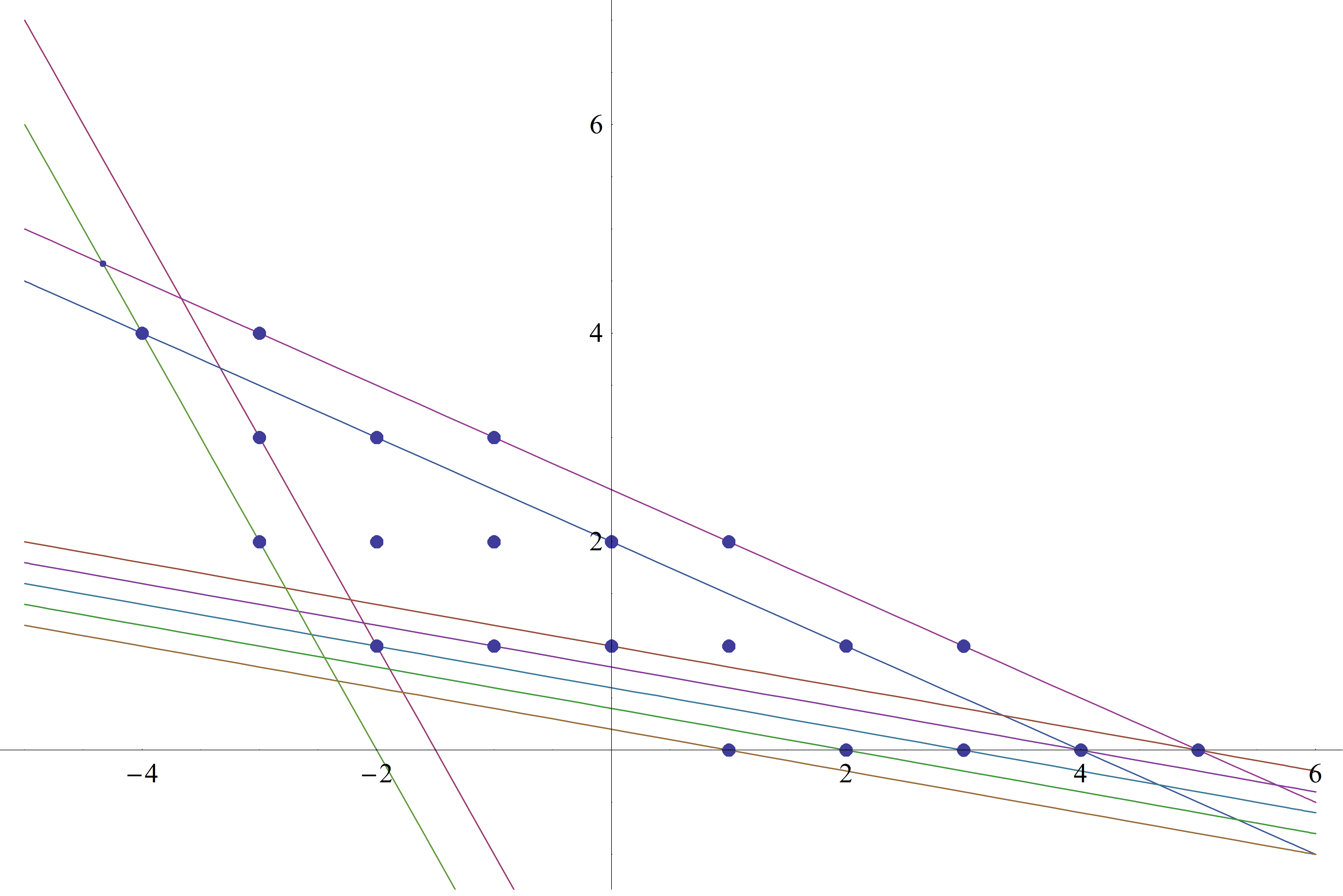}
\quad
\includegraphics[width=7cm]{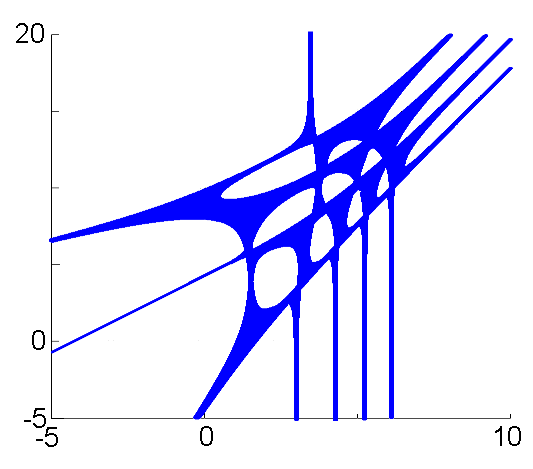}
\caption{The support of $p_{2}(x,y)$ (bounded by the singular
divisors of~(\ref{moreInvolvedOSCoeff})) and its amoeba}
\label{fig:simplicialPolyAndAmoeba}
\end{figure}

\begin{example}\label{ex:HGSimplicialComplex}\rm
The bivariate hypergeometric polynomial supported in the
quadrilateral with the vertices $(2,0),(3,2),(2,3)$ and $(0,1)$ is
given by $p_{3}(x,y) = 240 x^2 + 3 y + 240 x y + 1080 x^2 y + 30 x
y^2 + 180 x^2 y^2 + 36 x^3 y^2 + 2 x^2 y^3.$
Fig.~\ref{fig:quadrilateral} {\it(a-c)} show the affine
amoeba~$\mathcal{A}_{p_3}$, the compactified amoeba of
$p_{3}(x,y)$, the weighted compactified amoeba of the 6th Hadamard
power of $p_{3}(x,y).$ Fig.~\ref{fig:quadrilateral}~{\it(d)} shows
the vanishing connected component with the order $(2,2)$ in the
complement of the weighted compactified amoeba of a deformed
version of $p_{3}(x,y).$ We remark that the small component
vanishes exactly at the point with the coordinates $(2,2),$ that
is, at the order of this component.

\begin{figure}[ht]
\begin{minipage}[h]{0.24\linewidth}
\center{\includegraphics[width=0.99\linewidth]{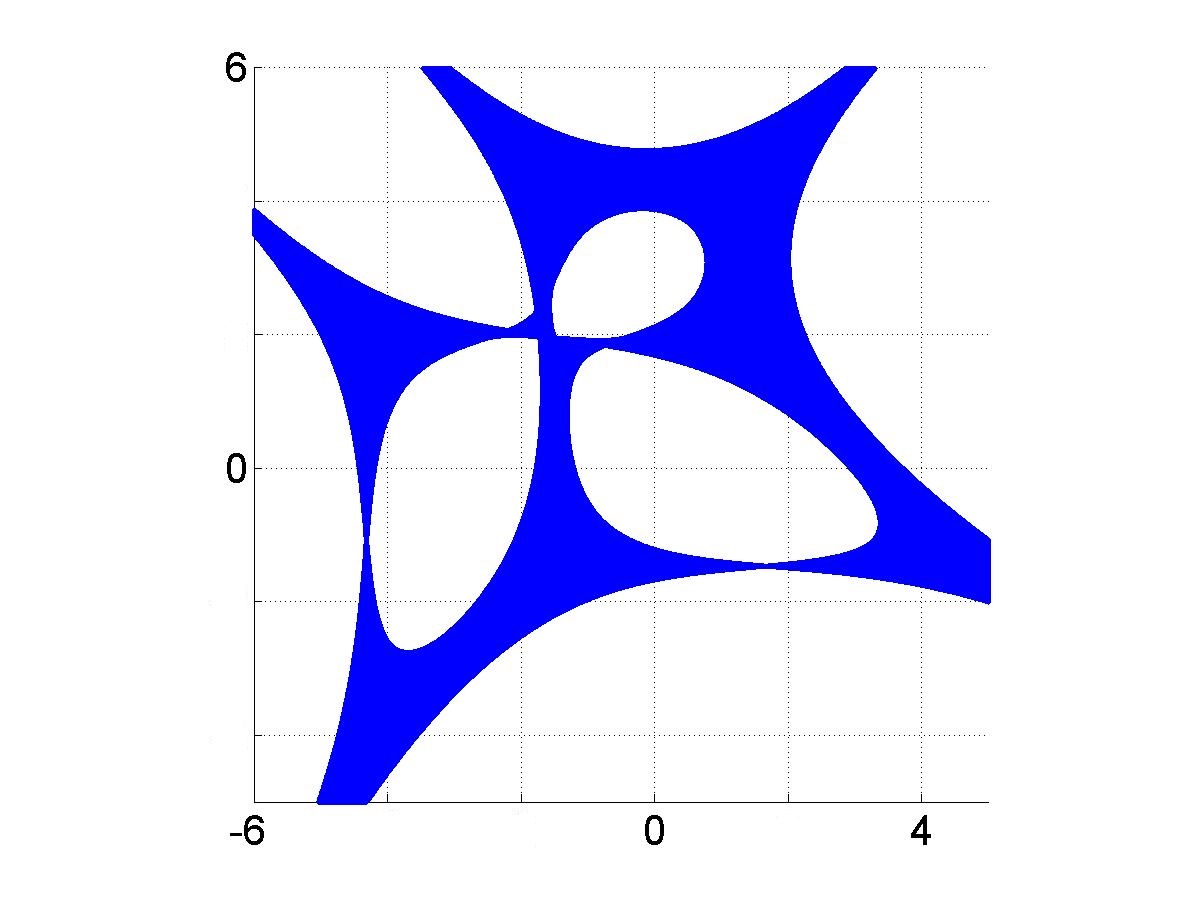} \\ (a)}
\end{minipage}
\begin{minipage}[h]{0.24\linewidth}
\center{\includegraphics[width=0.99\linewidth]{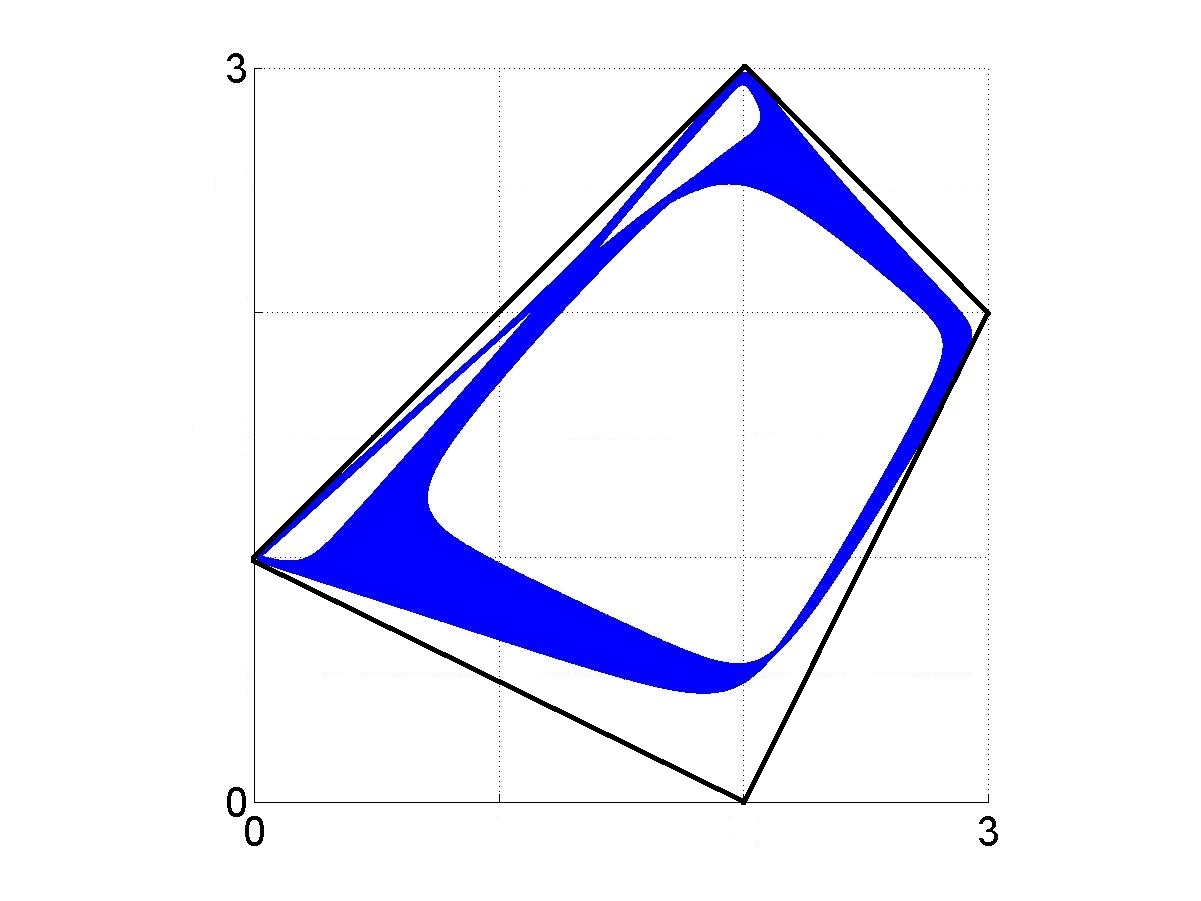} \\ (b)}
\end{minipage}
\begin{minipage}[h]{0.24\linewidth}
\center{\includegraphics[width=0.99\linewidth]{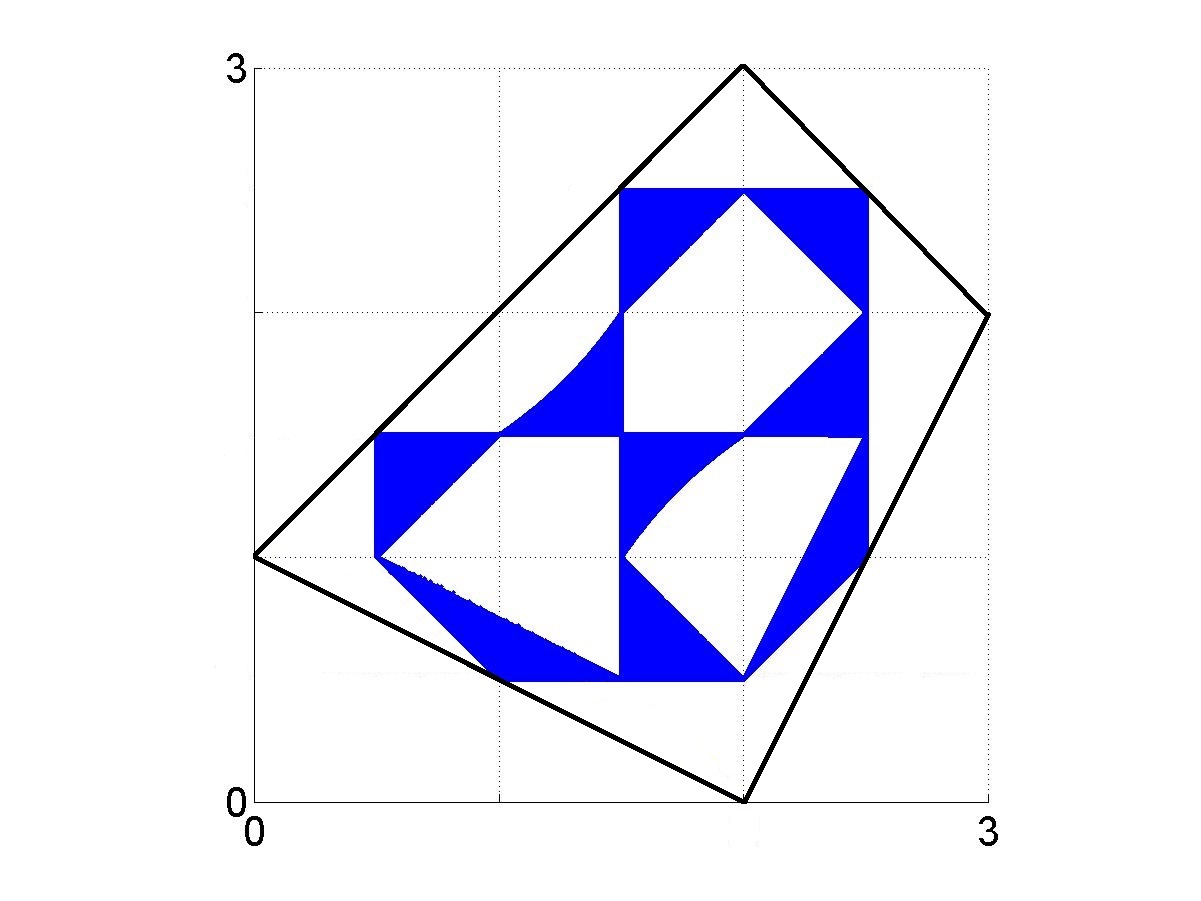} \\ (c)}
\end{minipage}
\begin{minipage}[h]{0.24\linewidth}
\center{\includegraphics[width=0.99\linewidth]{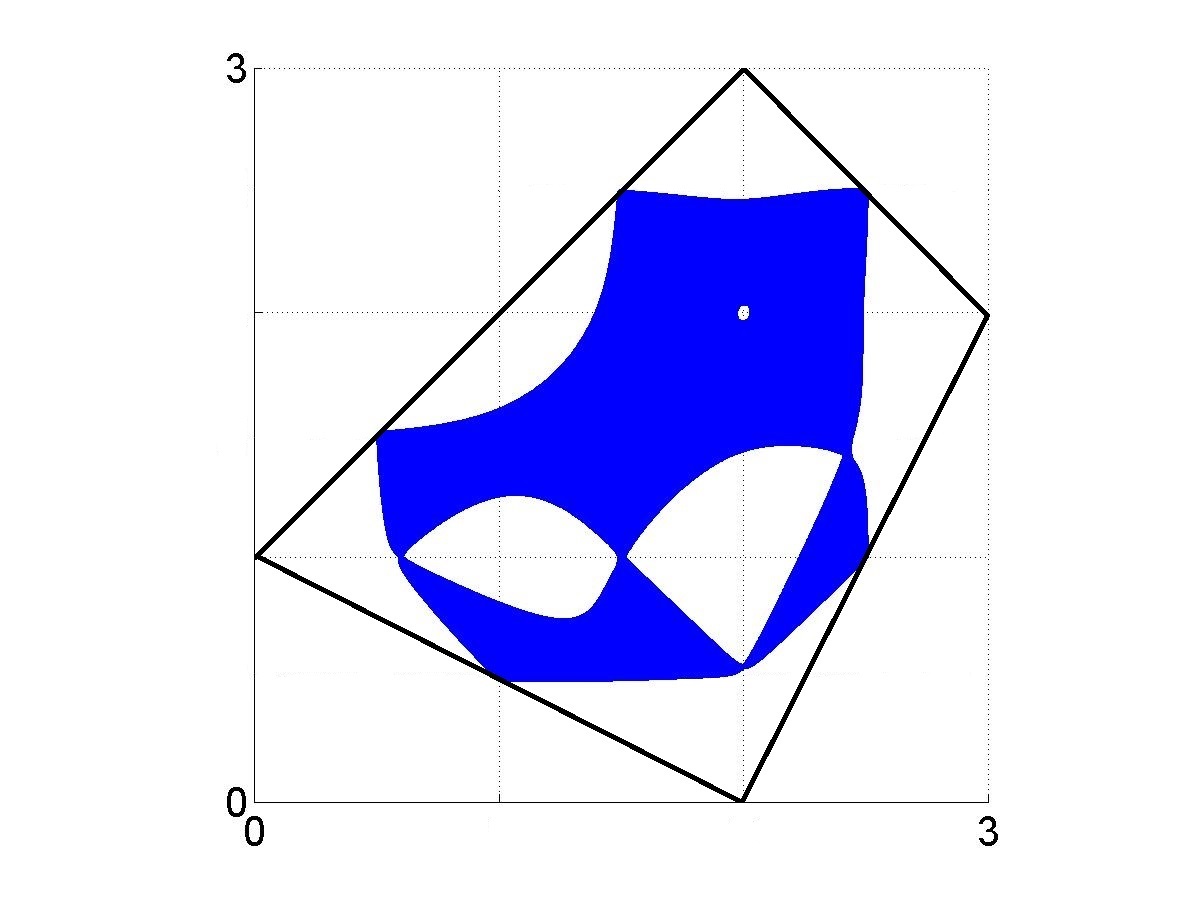} \\ (d)}
\end{minipage}
\caption{The amoeba, compactified amoeba, the weighted
compactified amoeba of the 6th Hadamard power and a vanishing
connected component of the complement to the
$\mathcal{WCA}(p_{3}(x,y))$} \label{fig:quadrilateral}
\end{figure}

\end{example}

\begin{example}\rm
The first maximal minor of the Toeplitz matrix
$$
\left(
\begin{array}{ccccccc}
 x & y & 1 & 0 & 0 & 0 & 0 \\
 1 & x & y & 1 & 0 & 0 & 0 \\
 0 & 1 & x & y & 1 & 0 & 0 \\
 0 & 0 & 1 & x & y & 1 & 0 \\
 0 & 0 & 0 & 1 & x & y & 1 \\
 0 & 0 & 0 & 0 & 1 & x & y
\end{array}
\right)
$$
is the degree~6 bivariate Chebyshev polynomial of the second
kind~\cite{AlexanderssonShapiro}. It is optimal in the coordinates
$\xi=xy,\,\eta=y^2/x$ which make it dense.
\end{example}



\begin{thebibliography}{99}
{\small

\bibitem{AlexanderssonShapiro}
P.\,Alexandersson and B.\,Shapiro. {\it Around a multivariate
Schmidt-Spitzer theorem}, Linear Alg. Appl.~{\bf 446}, no.~1
(2014),~356-368.

\bibitem{Beukers2009}
F.\,Beukers. {\it Algebraic A-hypergeometric functions}, Invent.
Math.~{\bf 180}, no.~3 (2010),~589-610.

\bibitem{Dominici} D.\,Dominici, S.J.\,Johnston, and K.\,Jordaan.
{\it Real zeros of $_2F_1$ hypergeometric polynomials}, Journal of
Comput. and Appl. Math.~{\bf 247} (2013), 152-161.

\bibitem{DriverJohnston}
K.A.\,Driver and S.J.\,Johnston. {\it Asymptotic zero distribution
of a class of hypergeometric polynomials}, Quaestiones
Mathematicae~{\bf 30}, no.~2 (2007), 219-230.

\bibitem{DunklXu}
C.F.\,Dunkl and Y.\,Xu. {\it Orthogonal Polynomials of Several
Variables.} Cambridge University Press, 2014.

\bibitem{Erdelyi}
A.\,Erdelyi. {\it Hypergeometric functions of two variables}, Acta
Math.~{\bf 83} (1950), 131-164.

\bibitem{FPT}
M.\,Forsberg, M.\,Passare, and A.\,K.\,Tsikh. {\it Laurent
determinants and arrangements of hyperplane amoebas}, Adv.
Math.~{\bf 151} (2000), 45-70.

\bibitem{GGR}
I.M.\,Gelfand, M.I.\,Graev, and V.S.\,Retach. {\it General
hypergeometric systems of equations and series of hypergeometric
type}, Russian Math. Surveys~{\bf 47}, no.~4 (1992),~1-88.

\bibitem{Guillemin}
V.\,Guillemin and S.\,Sternberg. {\it Convexity properties of the
moment mapping}, Invent. Math.~{\bf 67}, no.~3 (1982),~491-513.

\bibitem{Klein}
F.\,Klein. {\it \"Uber die Nullstellen der hypergeometrischen
Reihe}, (German) Math. Ann.~{\bf 37}, no.~4 (1890), 573-590.

\bibitem{Mikhalkin}
G.\,Mikhalkin. {\it Real algebraic curves, the moment map and
amoebas}, Ann. Math.~(2)~{\bf 151}, (2000) 309-326.

\bibitem{Norlund}
N.E.\,N{\o}rlund. {\it Hypergeometric functions}, Acta Math.~{\bf
94} (1955), 289-349.

\bibitem{PST}
M.\,Passare, T.M.\,Sadykov, and A.K.\,Tsikh. {\it Nonconfluent
hypergeometric functions in several variables and their
singularities}, Compos. Math.~{\bf 141}, no.~3 (2005), 787-810.

\bibitem{Purbhoo}
K.\,Purbhoo. {\it A Nullstellensatz for amoebas}, Duke
Math.~J.~{\bf 141}, no.~3 (2008), 407-445.

\bibitem{Rullgaard}
H.\,Rullg\aa rd. {\it Stratification des espaces de polyn\^{o}mes
de Laurent et la structure de leurs amibes} (French), Comptes
Rendus de l'Academie des Sciences - Series I: Mathematics~{\bf
331}, no.~5 (2000), 355-358.

\bibitem{Sadykov-SMZh}
T.M.\,Sadykov. {\it On a multidimensional system of hypergeometric
differential equations}, Siberian Math.~J.~{\bf 39} (1998),
986-997.

\bibitem{Sadykov-MathScand}
T.M.\,Sadykov. {\it On the Horn system of partial differential
equations and series of hypergeometric type}, Math. Scand.~{\bf
91} (2002),~127-149.

\bibitem{SadykovTanabe}
T.M.\,Sadykov and S.\,Tanabe. {\it Maximally reducible monodromy
of bivariate hypergeometric systems}, Izv. Math.~{\bf 80}, no.~1,
(2016), 221-262.

\bibitem{Theobald} T.\,Theobald and T.~de~Wolff. {\it Amoebas of genus at most
one}, Adv. Math.~{\bf 239} (2013), 190-213.

\bibitem{Zharkov} I.\,Zharkov. {\it Torus fibrations of Calabi-Yau hypersurfaces in toric
varieties}, Duke Math.~J.~{\bf 101}, no.~2 (2000), 237-257.

\bibitem{ZhouSrivastavaWang}
J.-R.\,Zhou, H.M.\,Srivastava, and Z.-G.\,Wang. {\it Asymptotic
distribution of the zeros of a family of hypergeometric
polynomials}, Proc. of the AMS,~{\bf 140}, no.~7 (2012),
2333-2346.

} 

\end{thebibliography}
\end{document}